%% file: hamiltongames-new2.tex
\documentclass {article}
\usepackage{graphicx,etex}
\usepackage{amssymb,amsthm,amsmath}
\usepackage{epsfig}
\input prepicte.tex
\input pictex.tex
\input postpict.tex

\def\text#1{\mbox{#1}}

\newtheorem {theorem}{Theorem}[section]
\newtheorem {definition}[theorem]{Definition}
\newtheorem {lemma}[theorem]{Lemma}
\newtheorem {question}[theorem]{Question}
\newtheorem {remark}[theorem]{Remark}
\newtheorem {prop}[theorem]{Proposition}
\newtheorem {corollary}[theorem]{Corollary}
\newtheorem {assumption}[theorem]{Assumption}

\newtheorem*{maintheorem}{Main Theorem}
\newtheorem*{2ndmaintheorem}{Second Main Theorem}
\newtheorem*{mainexample1}{Examples having interesting first return maps} 
\newtheorem*{mainexample2}{Example having random walk-like behaviour} 
\newtheorem*{definitionno}{Definition}
\newtheorem*{conjectureno}{Conjecture}

\newenvironment{proofof}[1]{\medskip 
\noindent{\it Proof of #1.}}{ \hfill\qed\\ }

\renewcommand{\rho}{\varrho}

\newcommand\dist{\mbox{\rm dist}}

\newcommand\M{\, {\mathcal M}\, }
\newcommand\MS{\, {\mathcal M}^*\, }
\newcommand\MB{\, {\mathcal M}^\bullet\, }

\newcommand\rz{{\mathbb R}}
\newcommand\R{{\mathbb R}}

\newcommand\N{{\mathbb N}}

\newcommand\bp{{\bar p}}
\newcommand\bq{{\bar q}}
\newcommand{\ii}{\underline 1}

\DeclareMathOperator*{\codim}{codim}
\DeclareMathOperator*{\argmax}{arg\,max}
\DeclareMathOperator*{\argmin}{arg\,min}
\newcommand\BR{{\mathcal B \! \mathcal R \!}}

\begin{document}
\title{Hamiltonian flows with random-walk 
behavior \\originating from zero-sum games and fictitious play}
\author{Sebastian van Strien}
\maketitle 

\begin{center}To Jacob Palis on his 70th birthday
\end{center}
\begin{abstract}
In this paper we introduce Hamiltonian dynamics,  inspired by zero-sum games 
(best response and fictitious play dynamics). 
Taking any piecewise affine Hamiltonian
of a very simple form, the corresponding Hamiltonian vector field
we obtain is discontinuous and multivalued. Nevertheless, somewhat surprisingly, 
the corresponding flow exists, is unique and continuous. 
We believe that these vector fields deserve attention, because 
it turns out that the resulting dynamics are rather different from those found in 
more classically defined Hamiltonian dynamics. The vector field
is extremely simple: it is piecewise constant and so the flow $\phi_t$ piecewise
a translation and in particular has no stationary points.  Even so, the dynamics 
can be rather rich and complicated. For example, 
there exist Hamiltonian vector fields on $\R^4$ of this type with energy level 
sets  homeomorphic to $S^3$  and so that the following holds.
There exists a  periodic orbit $\Gamma$  of the Hamiltonian flow so that,
restricting to the level set containing $\Gamma$,  the first return map $F$ to a two-dimensional
section $Z$  transversal to $\Gamma$ at $x\in \Gamma$
acts as a {\em random-walk}:  there exists a nested sequence of annuli $A_n$ in $Z$ (around $x$ so that  $\cup A_n\cup \{x\}$ is a neighbourhood of $x$ in $Z$)  shrinking geometrically
to $x$  so that  for each sequence $n(i)\ge 0$ with $|n(i+1)-n(i)|\le 1$ 
there exists a point $z\in Z$ so that $F^i(z)\in A_{n(i)}$ for all $i\ge 0$.
These Hamiltonian vector fields can also be used to obtain interesting continuous, area preserving,
piecewise affine maps $R$ on certain two-dimensional  polygons $S$,
for which $R|\partial S=id$ so that $R$ has infinitely many periodic points
(and conjecturally periodic orbits are dense on certain open subsets of $S$).
In the last two sections of the paper we give some applications to game theory, and finish with posing a version of the Palis conjecture in the context of the
class of non-smooth systems studied in this paper. 
\end{abstract}


\section{Introduction}
In this paper we will introduce a rather unusual class of Hamiltonian systems.
These are motivated by dynamics, usually referred to as {\lq}fictitious play{\rq} and {\lq}best response dynamics{\rq}, associated to zero-sum games.
These Hamiltonian systems differ from those that are considered
traditionally, in that  their orbits consist of piecewise straight lines and first return maps
to certain planes are piecewise isometries.
Specifically, the Hamiltonian systems we consider are continuous and piecewise 
affine, and defined by the Hamiltonian  $H\colon \Sigma\times \Sigma\to \R$,
\begin{equation*}
H(p,q)= \max_{p\in \Sigma}  p' \M  q -  \min_{q\in \Sigma} p'  \M q.
\end{equation*}
Here $\M$ is a $n\times n$ matrix, $p,q\in \Sigma$, the set of 
probability vectors in $\rz^n$ and $p'$  stands for the transpose of $p$.
Note that  $\max_{p\in \Sigma}\,   p' \M  q$ is equal to the largest component(s)
of the column vector $\M q$ and, similarly, $\min_{q\in \Sigma} \, p'  \M q$ is equal to 
the smallest component(s) of the row vector $p'\M$. 
In other words, 
$$
H(p,q)= \max_i \, (\M  q)_i -  \min_j \, (p'  \M)_j\,\, ,  
$$
where $(\M q)_i$ and $(p' \M)_j$ stands for the $i$-th and $j$-th component
of the vectors $\M q$ respectively $p' \M$.
Hence $H$ is continuous and piecewise affine 
($H$ is affine outside some finite union of hyperplanes). 

\begin{definitionno}[Completely mixed Nash equilibrium]
We say that a $n\times n$ matrix $\M$ has an {\em completely mixed Nash equilibrium}
if there exist a unique $\bp,\bq\in \Sigma$ so that all its components are strictly positive
and so that $\bp' \M=\lambda \ii$ and $\M \bq=\mu \ii$ for some $\lambda,\mu\in \R$
where $\ii=(1\, 1 \, \dots \, 1)\in \R^n$. \end{definitionno}

Let us denote the set of all $n\times n$ matrices by $L^n$ and the set
of $n\times n$ matrices with a completely mixed Nash equilibrium by $LI^n$. 
Clearly $LI^n$ is an open subset of $L^n$.

\begin{maintheorem} There exists a subset $V\subset LI^n\times L^n$
which is open and dense and has full Lebesgue, so that for each pair
of $n\times n$ matrices
$(\M,A)\in V$ the following holds:
\begin{enumerate}
\item Any level 
set $H^{-1}(\rho)$, for $\rho>0$ small, is topologically a $2n-3$-sphere made up of hyperplanes.
(Note that the dimension of $\Sigma\times \Sigma$ is $2n-2$.)
\item The Hamiltonian vector field $X_H$ associated to $H$ and the symplectic  2-form
$\sum_{ij}a_{ij} dp_i\wedge dq_j$ (where $(a_{ij})$ are the coefficients of the matrix $A$)
is piecewise constant and set-valued in codimension-one sets.
\item Correspondingly, we have a differential inclusion
$$(\dfrac{dp}{dt},\dfrac{dq}{dt})\in X_H(p,q)$$
where the right hand side is set-valued and $(p,q)\mapsto X_H(p,q)$ is piecewise constant.
\item The differential inclusion induces a unique continuous flow on 
$H^{-1}(\rho)$ (for each $\rho>0$ fixed and small) which is piecewise a translation flow.
The flow has no stationary point.
\item First return maps between hyperplanes are piecewise affine maps.
\end{enumerate}
\end{maintheorem}

\medskip
 
 Although we are dealing with a differential inclusion,
 solutions are unique, and define a continuous flow!
 Moreover, although these Hamiltonian vector fields are very simple in that they
are piecewise constant, they generate surprisingly rich dynamics. 
Let us illustrate this by describing some interesting properties.

\subsection{Examples with interesting properties}

The first of property shows that many of these Hamiltonian vector fields induce
continuous area-preserving piecewise affine maps $R$ on polygons $S$ in $\R^2$ so that 
$R=id$ on $\partial S$ and so that $R$ has infinitely many periodic orbits. 
We conjecture that periodic orbits are dense in open subsets of $S$. To construct 
maps of this type by {\lq}hand{\rq} seems not so easy.

\begin{mainexample1}\label{mainexample1}
There exists an open set of matrices $(\M,A)\in LI^3\times L^3$ so that for each  
corresponding Hamiltonian vector field $X_H$ there exists a topological disc $S$ consisting of four 
triangles (which are embedded in $\R^4$), see 
Figure~\ref{fig:S},  
so that the first return map $R$ to $S$ of the flow of $X_H$  has the following properties:
\begin{enumerate}
\item $R\colon S\to S$ is area-preserving (w.r.t. Lebesgue measure), continuous and piecewise affine;
\item $R=id$ on $\partial S$;
\item $R$ has infinitely many periodic orbits of saddle-type.
\end{enumerate}
\end{mainexample1}

\begin{figure}[htp]
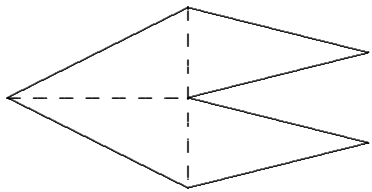
 \hfil
\beginpicture
\dimen0=0.12cm
\dimen1=0.06cm
\setcoordinatesystem units <\dimen0,\dimen1> point at 30 0 
\setplotarea x from 0 to 40, y from -20 to 20
\setlinear
\plot 0 0  20 20  40 10  20 0 40 -10 20 -20 0 0 /
\setdashes
\plot 20 20 20 -20 /
\plot 0 0 20 0 / 
\endpicture
\caption{\label{fig:S}
{\small The topological disc $S$ consisting of four triangles.}}
\end{figure}

\begin{conjectureno} Let $R\colon S\to S$ be as in the previous set
of examples. Then there exists a compact set $K\subset S$ so that the following properties hold.
\begin{itemize}
\item $K$ is either empty or consists of two regions bounded
by two disjoint ellipses in $S$.
\item If $K$ is non-empty, then $R$ permutes the two components
of $K$. Moreover, $R^2|K$ is linearly conjugate to a rotation.
\item Periodic orbits are dense in $S\setminus K$.
\item There exists a dense orbit in $S\setminus K$
and $R$ is  ergodic on $S\setminus K$.
\end{itemize}
\end{conjectureno}

Simulations give dynamics as in Figure~\ref{fig:S-triangles}. 
Although {\lq}typical{\rq} orbits seem to be dense outside the two elliptic regions, 
the mechanism could be rather different from that of  Arnol'd diffusion in classical
smooth area preserving maps, see \cite{MR0163026}. 

Another difference with classical dynamical systems is that
invariant manifolds associated to a periodic orbit $\Gamma$ of the flow can be rather strange.
For a smooth system one expects that all orbits near $\Gamma$ approach
$\Gamma$ (either in forward or backward time) with some rate $\mu$
where $\mu$ is one of  the multipliers of the linearization of a first return map along $\Gamma$. 
So only a finite number of rates is possible. Here the situation is totally different:

\begin{mainexample2}
There exists an open set of matrices $(\M,A)\in LI^3\times L^3$
so that for the corresponding Hamiltonian vector field there
exists a periodic orbit $\Gamma$ so that
the stable manifold $W^\tau(\Gamma)$ of speed $\tau$ is non-empty for 
{\em each} $\tau\approx 0$.
\end{mainexample2}

So for {\em each} $\tau\approx 0$ there exist an initial condition $(p,q)$ so that 
$$d(\phi_t(p,q),\Gamma)\approx e^{\tau t} \quad   \left\{ \begin{array}{ll}
\mbox{  as }t\to \infty & \mbox{ when }\tau<0,\\ 
\mbox{  as }t\to -\infty &\mbox{ when }\tau>0.\end{array} \right.
$$
In Section~\ref{sec:randomwalk} we will explain this in more detail, 
and relate it to random-walk like behaviour. 

\subsection{Relationship to game theory}

In the last section of this paper we will show how the previous set-up is
related to game theory. However, let us  state here already the following:

\begin{2ndmaintheorem} There exists a subset $W\subset LI^n$
which is open and dense and has full Lebesgue, so that for each pair
of $n\times n$ matrix
$\M\in W$ the following holds:
\begin{enumerate}
\item $(\bp,\bq)\in \Sigma\times \Sigma$  is the unique Nash equilibrium 
of the zero-sum game best response dynamics 
\begin{equation}
\dfrac{dp}{dt}\in \BR_p(q)-p, \dfrac{dq}{dt}\in \BR_q(p) -q. \label{eq:brorig}
\end{equation}
associated  to $\M$. Here
$\BR_p(q):=\argmax_{p\in \Sigma}  p'\M q$ and 
$BR_q(p):=\argmin_{q\in \Sigma} \, p' \M q$.
\item The flow $\phi_t$ associated to this differential inclusion exists, is unique and 
continuous outside $(\bp,\bq)$.
\item For each $\rho>0$ small, 
each half-ray through $(\bp,\bq)$ intersects $H^{-1}(\rho)$ in a unique point.
Let $\pi\colon (\Sigma\times \Sigma)\setminus \{(\bp,\bq)\}\to H^{-1}(\rho)$
be the corresponding continuous map. 
Then $\pi\circ \phi_t$ is the flow on $H^{-1}(\rho)$
of a Hamiltonian vector field corresponding to $H$ and the symplectic form $\sum_{ij}a_{ij}dp_i\wedge dq_j$ where
$(a_{ij})$ are the coefficient of the matrix $A:=\M$. 
\end{enumerate}
\end{2ndmaintheorem}

In fact, we will prove this result for a more general differential inclusion, see equation (\ref{eq:genBR}).


\subsection{Relationship with other papers}
Before describing these results in more detail,  
let us give a bit of background to the result
of this paper and how it relates to the literature.

This study grew out of research into the best response dynamics (\ref{eq:brorig}) associated to 
two player games. This differential inclusion, or rather
the differential inclusion 
 \begin{equation}
\dfrac{dp}{dt}\in \dfrac{1}{t}(\BR_p(q)-p), \dfrac{dq}{dt}\in \dfrac{1}{t}(\BR_q(p) -q). \label{eq:fporig}
\end{equation}
was introduced in the late 1940's by Brown \cite{Brown49} to describe a mechanism 
in which two players could {\lq}learn to play a Nash equilibrium{\rq}, and since then usually 
is  referred to as {\em fictitious play dynamics}.
Note that the orbits of (\ref{eq:brorig}) and (\ref{eq:fporig}) are the same up to time reparametrization
and that  their right hand side is piecewise affine (and multivalued in places).
It the early 50's Robinson~\cite{Robinson51} showed 
that the solutions of these differential inclusions converge
to the set of Nash equilibria of the game.  One can also define
the corresponding differential inclusion for non-zero games,  see Section~\ref{sec:games}. 
However,  in the 60's Shapley \cite{Shapley64} showed that 
in that case these equations can have periodic attractors. Shapley's example is extremely
well-known in the very extensive literature on fictitious games, for references
see for example \cite{SSH08} and for a discussion on the relationship of fictitious play
and learning, see \cite{Fudenberg-Levine98}.
In the first of a sequence of papers  (\cite{SSH08} and \cite{SS09}), 
we considered a one-parameter family of games, which includes Shapley's
classical example, and analysed how Shapley's periodic orbit bifurcates and how eventually
other `simple' periodic orbits are created 
In the second paper  
we study the dynamics
of these games in much more detail, and how one can have `chaotic choice of strategies' for the players.
One of the ideas in that paper is to study the dynamics induced by
projecting orbits onto the boundary of the space. 
In numerical studies, we observed (many years ago) that,
in the zero-sum case, this induced dynamics behaves similarly to that of an area preserving flow. The
present paper explains this observation.

In \cite{SS09} we observed 
that the  transition map associated to 
differential inclusion  between hyperplanes is a composition of piecewise projective maps.
In fact, as we show in this paper in the present case   the transition map 
is continuous, volume preserving and  piecewise affine.  In particular, we obtain a family of 
continuous, piecewise affine,
area preserving  maps of a polygon in the plane with rather interesting dynamics. 
This connects this paper with an exciting body of work on
piecewise isometries (with papers by R. Adler, P.Ashwin,  M. Boshernitzan, A. Goetz, B. Kichens, 
T. Nowicki,  A. Quas, C. Tresser and many others). Most of these papers deal with
piecewise continuous maps, while the maps we encounter are continuous. 

Another loose connection of our work is to that of the huge and very active field
of translation flows (associated to interval exchange transformations, translation surfaces
and Teichm\"uller flows) (with recent papers by A. Avila, Y. Cheung, A. Eskin, G. Forni, P. Hubert, 
H. Masur, C. McMullen,  M. Viana, J-C. Yoccoz, A. Zorich and many many others). But of course
our flow does not act  on a surface with a hyperbolic metric, and so this connection seems not very
helpful. 

\subsection{Relationship with the literature on non-smooth dynamical systems}
We should point out that there are many results on nonsmooth dynamical systems. 
Most of these 
are motivated by mechanical systems with {\lq}dry friction{\rq}, {\lq}sliding{\rq}, {\lq}impact{\rq} and so on.
As the number of workers in this field is enormous,  we just refer to the recent survey of M. di Bernardo et al 
\cite{Bernardo_etal09} and the monograph by M. Kunze \cite{Kunze2000}.  Of course our paper is very much related to this work, although
the motivation and the result seem to be of a different nature from what can be found in those papers.  

What our results have in common with many of the models in this literature, is that we are dealing
with a differential inclusion 
$x'\in f(x)$ where $f(x)$ is discontinuous and multi-valued in some hyperplanes. 
Our paper
deals with  a situation in which we also study differential inclusions but
for which the flow exists, is unique and continuous. 
 Moreover, the global dynamics around certain periodic orbits can be extremely complicated (much more so than would be possible in the smooth case). This behaviour is described in Theorem~\ref{thmB}. 
 We believe that this behaviour has not been observed before. 
%
%

\subsection{The organization of this paper}  
In Section 2 and 3 we discuss why level sets of $H$ are spheres 
and how to compute the Hamiltonian vector field $X_H$.
In Section 4 we show that a related (set-valued) vector field  has a 
continuous flow. In Section 5, we show this implies that the flow
of the original Hamiltonian vector field is continuous. In Section~\ref{sec:example} we will discuss
some examples. Finally, in Section 8 we relate these results to 
dynamics associated to  game theory (the so-called best response
and fictitious play dynamics).
We should emphasize that this paper does not include the proofs related
to the examples for which random walk behaviour is shown. We describe
these results briefly in Section~\ref{sec:example}, but for a detailed analysis  see  \cite{SS09}.

This paper requires no knowledge of game theory.

\subsection{Notation and terminology}
If $b_1,\dots,b_k\in \R^n$ then we will denote
by $<b_1,\dots,b_k>$ the space spanned by these vectors
and $[b_1,\dots,b_k]$ the space of convex combinations of these points. 
We also sometimes refer to the {\em subspace} associated to  $[b_1,\dots,b_k]$:
this is the smallest affine space containing $[b_1,\dots,b_k]$. 
All vectors we consider are column vectors. If $p\in \R^n$
then $p'$ always denotes the corresponding row vector.
The transpose of a matrix $M$ is denoted by $M'$. 


\section{$H$ and the best response functions}
\label{sec:ThmsHamilton}
Let $n,m$ be positive integers and consider $H\colon \Sigma_p\times \Sigma_q\to \R$  
of the form
\begin{equation}
H(p,q)= \max_{p\in \Sigma}  p' \M  q -  \min_{q\in \Sigma} p'  \M q,
\label{eq:hamilton}
\end{equation} 
where $\Sigma_p$ and $\Sigma_q$ are the set of probability vectors  in $\R^m$ resp. $\R^n$ 
and
$\M$ is a $m\times n$ matrix. 
We will denote by $e_1^p,e_2^p,\dots,e_m^p$  the unit vectors
in $\Sigma_p$ and by $e_1^q,e_2^q,\dots,e_n^q$ the unit vectors
in $\Sigma_q$.  Often we will consider the case that $m=n$ and then write $\Sigma=\Sigma_p=\Sigma_q$.
We will denote the subspace spanned by $\Sigma_p$ and $\Sigma_q$ by 
$\hat \Sigma_p$ and $\hat \Sigma_q$ and their tangent spaces by  $T\Sigma_p$ and $T\Sigma_q$.

One can also write
\begin{equation}
H(p,q)=(\BR_p(q))' \M  q- p' \M \BR_q(p). \label{eq:rewriteH}
\end{equation}
where $\BR_p\colon \R^n \to \Sigma_p$ and $\BR_q\colon \R^m \to \Sigma_q$
are defined by
\begin{equation}
\BR_p(q):=\argmax_{p\in \Sigma_p}  p'\M q \,\, 
\mbox{ and }\,\,  \BR_q(p):=\argmin_{q\in \Sigma_q} \, p' \M q .
\label{eq:brpq}
\end{equation}
Due to their game-theoretic interpretation, which we will discuss in Section~\ref{sec:games}, 
$\BR_p$ and $\BR_q$ are called {\em best response functions}. 
Note that  $\BR_p$ and $\BR_q$ are in actual fact set-valued, because 
$\BR_p(q)=[e_{i_1}^p,\dots,e_{i_k}^p]$ iff $i_1,\dots,i_k$
are the largest components of $\M q$. Similarly,   $\BR_q(p)$ 
is equal to the convex combination of unit vectors corresponding
to the smallest components of $p \M$. It follows that 
$$\R^n \ni q\mapsto \BR_p(q)\in \Sigma_p\mbox{ and }\R^m \ni p\mapsto \BR_q(p)\in \Sigma_q$$ are set-valued
and upper semi-continuous.  
For example, when $m=n=2$ and $\M$ is the $2\times 2$ identity matrix, then we can 
write the probability vectors $p,q$ in the form $p=(p_1,1-p_1)$ and $q=(q_1,1-q_1)$,
identify $\Sigma\times \Sigma=[0,1]\times [0,1]$ and we get
$H(p,q)=\max(|q_1|,|q_2|)-\min(|p_1|,|p_2|)$. Moreover $\BR_p(q)=e_1^p$ when $q_1>1/2$, 
$\BR_p(q)=e_2^p$ when $q_1<1/2$ and $\BR_p(q)=\Sigma_p$ when $q_1=1/2$. 
When $m=n=3$ and $\M$ is the $3\times 3$ identity matrix, 
then $H(p,q)= \max_i \,q_i -  \min_j \, p_j$ and $\BR_p$ and $\BR_q$ are shown in Figure~\ref{fig:levelsets} (on the right). 

\begin{figure}[htp]
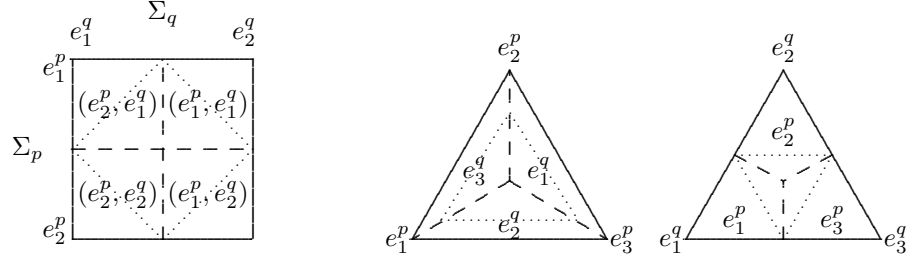
 \hfil
\beginpicture
\dimen0=0.12cm
\dimen1=0.13cm
\setcoordinatesystem units <\dimen0,\dimen0> point at 30 0 
\setplotarea x from 0 to 20, y from -2 to 20
\setlinear
\plot 0 0 20 0 20 20 0 20 0 0 /
\setdashes
\plot 10 0 10 20 /
\plot 0 10 20 10 /
\setdots <1mm>
\plot 0 10 10 0 20 10 10 20 0 10 / 
\put{$\Sigma_q$} at 10 25
\put{$\Sigma_p$} at -5 10
\put{$e_1^q$} at 1 23
\put{$e_2^q$} at 19 23
\put{$e_2^p$} at -2 1
\put{$e_1^p$} at -2 19
\put{$(e_2^p,e_2^q)$} at 5 5 
\put{$(e_1^p,e_2^q)$} at 15 5 
\put{$(e_1^p,e_1^q)$} at 15 15 
\put{$(e_2^p,e_1^q)$} at 5 15 
\setcoordinatesystem units <\dimen1,\dimen1> point at -35 0 
\setplotarea x from 0 to 20, y from -2 to 20
\setsolid \setlinear
\plot 0 0 20 0 /
\plot 0 0 10 17.3 20 0 /
\setdashes
\plot 10 0 10 6 /
\plot 5 8.66 10 6 /
\plot 15 8.66 10 6 /
\setdots <1mm>
\plot 10 0 5 8.66 15 8.66 10 0 /
\setsolid
\put {$e^q_1$} at -1.5 0
\put {$e^q_2$} at 10 19.6
\put {$e^q_3$} at 21.5 0
\put {{$e_1^p$}} at 5 2
\put {{$e_3^p$}} at 15 2
\put {{$e_2^p$}} at 10 11
\setcoordinatesystem units <\dimen1,\dimen1> point at -7 0
\setplotarea x from 0 to 30, y from -2 to 20
\setsolid
\setlinear
\plot 0 0 20 0 /
\plot 0 0 10 17.3 20 0 /
\setdashes
\plot 10 6 0 0 /
\plot 10 6 20 0 /
\plot 10 6 10 17.3 /
\setdots <1mm>
\plot 3 2   17 2   10 13 3 2 /
\put {$e^p_1$} at -1.5 0
\put {$e^p_2$} at 10 19.6
\put {$e^p_3$} at 21.5 0
\put {{$e_2^q$}} at 10 1.5
\put {{$e_1^q$}} at 13 7
\put {{$e_3^q$}} at 6.4 7
\endpicture
\caption{\label{fig:levelsets}
{\small On the left a level set of $H$ (drawn in dots), where we consider the case that $n=2$, 
take $\M$ the $2\times 2$ identity matrix and identify $\Sigma\times \Sigma=[0,1]\times [0,1]$.
The values of $(\BR_q(p),\BR_p(q))$ are also shown.
In the right two figures the case where $m=n=3$, $\M=id$ is considered: 
level sets of $\max |q_i|$ and $\min|p_i|$ are drawn inside the simplexes $\Sigma_p$ and $\Sigma_q$. 
The values of $\BR_q(p)$ and $\BR_p(q)$ are also shown.}}
\end{figure}

\begin{definition}[Nash equilibrium]
We call $(\bp,\bq)\in \Sigma_p\times \Sigma_q$ a {\em Nash equilibrium} if 
$$\bp\in \BR_p(\bq)\mbox{ and }\bq\in \BR_q(\bp).$$
\end{definition}

It is well-known that associated to each matrix $\M$ there exists
at least one Nash equilibrium, see for example \cite[Proposition 20.3]{MR1301776}.

\begin{prop}
$H$ is continuous, 
$H\ge 0$ and $H(p,q)=0$ iff $(p,q)$ is a Nash equilibrium.
Moreover,  if $(\bp,\bq)$ is the only Nash equilibrium and all components
of $\bp,\bq$ are strictly positive then
 $H^{-1}(\rho)$ is a $(m+n-3)$-dimensional sphere in $\Sigma_p\times \Sigma_q$ for any $\rho>0$ small
and there exists a continuous map 
\begin{equation}
\pi\colon \Sigma_p\times \Sigma_q\setminus \{(\bp,\bq)\} \to H^{-1}(\rho)
\label{eq:pi}
\end{equation}
which maps  any point on a half-ray $l_+$ through $(\bp,\bq)$ to 
the point $l_+\cap H^{-1}(\rho)$. 
\end{prop}
\begin{proof} Continuity of $H$ is clear from the definition.
Note that $\BR_p(q)' \M  q \ge p' \M q \ge p' \M  \BR_q(p)$ for all $(p,q)\in \Sigma_p\times \Sigma_q$ and so $H\ge 0$. 
Moreover,  $H(p,q)=0$ iff both equalities hold. This is equivalent to $p\in \BR_p(q)$ and $q\in \BR_q(p)$
and therefore to $(p,q)$ being a Nash equilibrium.
Now take some Nash equilibrium $(\bp,\bq)$ and 
take a half-ray $l_+$  through $(\bar p,\bar q)$.
The function $H$ is  increasing on $l_+$. 
Indeed, take $(p,q)\in l_+$ and let $\lambda>0$.  Then
\begin{equation*}
\begin{array}{rl}
&\BR_p (q+ \lambda(q-\bq))' \M  (q+\lambda(q-\bq)) 
\ge \BR_p(q)' \M  (q+\lambda(q-\bq)) \\
&\quad =(1+\lambda)\BR_p(q)' \M q - \lambda\BR_p(q)'\M \bq \ge (1+\lambda)\BR_p(q)' \M q - \lambda \BR_p(\bq)'\M \bq\\
 & \quad = (1+\lambda)\BR_p(q)' \M q - \lambda \bp'\M \bq .
\end{array}
\end{equation*}
Similarly, 
$$(p+\lambda(p-\bp))' \M  \BR_q(p+\lambda(p-\bp)) \le 
(1+\lambda) p' \M \BR_q(p) - \lambda \bp'\M \bq.$$
It follows that 
$H(p+\lambda(p-\bp),q+\lambda(q-\bq))\ge (1+\lambda) H(p,q)$
and that the zero set of $H$ is convex.
So if we assume that $H$ has precisely one zero (in the interior of $\Sigma_p\times \Sigma_q$)
then $H$ is strictly increasing on $l_+$ and the set $H^{-1}(\rho)$ is contained
in $\Sigma_p\times \Sigma_q$ for $\rho>0$ small. 
By considering the map which assigns to each such a half-ray $l_+$ 
the point   $H^{-1}(\rho)\cap l_+$ we obtain a homeomorphism between a sphere in $\R^m\times \R^n$
centered at $(\bp,\bq)$ and the set $H^{-1}(\rho)$.
\end{proof}

\begin{remark}
The reason we require $\rho>0$ to be small, is because we only defined $H$  
on $\Sigma_p\times \Sigma_q$.  We could also have defined $H$ on the 
hyperplanes containing $\Sigma_p, \Sigma_q$ and consider the set $H^{-1}(1)$, but then 
some points in $H^{-1}(1)$ could have components which are not strictly positive.
\end{remark}

In this paper we will mainly consider matrices $\M$ for which there exists a
completely mixed Nash equilibrium: 

\begin{definition}[Completely mixed Nash equilibrium]\label{def:compmixed}
We say that the pair $(\bar p,\bar q)$ is a {\em completely mixed Nash equilibrium} if  
$(\bp,\bq)$ is  a Nash equilibrium and 
all components of $\bar p$ and $\bar q$ 
are strictly positive.
\end{definition}

This definition agrees with the previous definition:

\begin{lemma}
If $(\bar p,\bar q)$ is a completely mixed Nash equilibrium then
there exists $\lambda,\mu\in \R $ so that  
$$\bar p'\M=\lambda \ii_q', \M \bar q=\mu \ii_p,$$
where $\ii_p=\left(1 \, 1 \, \hdots \, 1 \right)\in \R^m$
and $\ii_q=\left(1 \, 1 \, \hdots \, 1 \right)\in \R^n$.
Reversely, if all components of $\bp$ and $\bq$ are strictly positive
and $\bar p'\M=\lambda \ii_q'$, $\M \bar q=\mu \ii_p$ then
$(\bp,\bq)$ is a completely mixed Nash equilibrium. 
\end{lemma}
\begin{proof} $\bar p\in \BR_p(\bar q)$ and all components
of $\bar p$ are strictly positive implies that all components of
$\M q$ are equal. 
This and the corresponding statement for $q$ implies the first part of the lemma.
The second part is obvious. 
\end{proof}

%

Assume the following:
\begin{assumption}[First transversality assumption]\label{eq:trans1} Every $r\times r$ minor of $\M$ is non-zero,   
$\forall r\ge 2$.
\end{assumption}
Here, as usual, we define a $r\times r$ {\em minor} of a matrix $A$ to be the determinant of a 
$r\times r$ matrix $B$ which is obtained by selecting only $r\le \min(m,n)$ of the rows and columns of $A$.
Note that if $n=m\ge 3$
then the $n\times n$ identity matrix does {\em not} satisfy the transversality condition.

\begin{prop}\label{prop:unique1}
If $\M$ satisfies the  transversality assumption (\ref{eq:trans1})
and if there exists a completely mixed Nash equilibrium
$(\bar p,\bar q)$, then $(\bp,\bq)$ is the unique Nash equilibrium.
\end{prop}
\begin{proof}
By the previous lemma, $\bp' \M=\lambda \ii_q$. By the  transversality assumption (\ref{eq:trans1}) for each $\lambda$
there exists a unique $\bar p$ satisfying this equation.
Since $\bp$ is a probability vector, this uniquely determines $\bp$
(among all completely  mixed Nash equilibria). Similarly, all components of $\M \bar q$ are equal,
and again $\bar q$ is uniquely determined (among all completely  mixed Nash equilibria).
On the other hand,  if not all coordinates of $\M q$ are equal then $\BR_p(q) \M q > \bp \M  q =\bp \M \bq$
(here $>$ holds because $\BR_p(q)$ necessarily puts no weight on some of the coordinates
since not all coordinates of $\M q$ are equal). Similarly, if 
not all coordinates of  $p \M$ are equal then $p \M \BR_q(p) < p \M \bq =\bp \M \bq$. 
Hence  $(p,q)\ne  (\bp,\bq)$ implies $H(p,q)>0$ and therefore by the
previous proposition that $(p,q)$ is not a Nash equilibrium.
\end{proof}

In some of the theorems we shall need to make an additional transversality assumption:

\begin{assumption}[Second transversality assumption]  \label{eq:trans2}
Every minor of $\MS$ is non-zero. Here $\MS$  is any   $(m-1)\times n$ (resp. $n\times (n-1)$) matrix
obtained by subtracting some row of $\M$ from all other rows of $\M$ 
(resp. some column from all other columns).
\end{assumption}

Note that this assumption requires in particular that all coefficients of the matrices $\MS$ 
are non-zero, and hence all coefficient within each row (and each column) 
vector of $\M$ are distinct.

We claim that this implies that certain spaces are in general position.

\begin{definition}[General Position]
Let $H^1,H^2$ be two affine subspaces of an affine space $\Sigma$.  We say that they are in {\em general position} 
w.r.t. $\Sigma$ if 
$$\mbox{ either }H^1\cap H^2=\emptyset\mbox{ or }H^1+H^2=\Sigma$$ 
or equivalently if
$$\mbox{ either }H^1\cap H^2=\emptyset \mbox{ or }\codim(H^1\cap H^2)=\codim(H^1)+\codim(H^2).$$
\end{definition}

\medskip

Note that from thie
definition  $\codim(H^1)+\codim(H^2)>\dim(\Sigma)$ implies $H^1\cap H^2=\emptyset$.
For example, two lines in $\R^3$ (and similarly  a point and a plane in $\R^3$)
are in general position w.r.t. $\R^3$  iff they do not meet.
The equivalence of the two definitions above holds because 
$\dim(H_1+H_2)$ is equal to \vskip -0.7cm
$$\dim(H_1)+\dim(H_2)-\dim(H_1\cap H_2)=\dim(\Sigma)-  
\left(\codim(H^1)+\codim(H^2)
 -\codim(H^1\cap H^2)\right).$$ 

\medskip

\begin{prop}\label{prop:unique2}  Assume the  transversality assumptions (\ref{eq:trans1}) and (\ref{eq:trans2}) are satisfied. 
Then 
\begin{enumerate}
\item $H$ is piecewise affine and
$\BR_p(q)$ and
$\BR_q(p)$  
 are constant outside codimension-one planes.
\item \label{itemb} Take any $1\le j_1<\dots<j_k\le n$
and consider the subspace $Z_p(j_1,\dots,j_k)\subset \Sigma_p$ 
of the set of $p\in \Sigma_p$ 
where $\BR_q(p)=[e^q_{j_1},\dots,e^q_{j_k}]$.
Moreover consider an  $r$-dimensional face
$[e_{i_1},\dots,e_{i_r}]$ of $\Sigma_p$.
Then the linear spaces spanned by $Z_p(j_1,\dots,j_k)$ and by $[e_{i_1},\dots,e_{i_r}]$
are in general position as subspaces of $\hat \Sigma_p$.
A similar statement holds when the roles of $p$ and $q$ are interchanged.
\item Suppose that $(\bp,\bq)$ is a Nash equilibrium. Then there exist $1\le r\le \min(m,n)$, 
$i_1,\dots,i_r\in \{1,\dots,m\}$
and $j_1,\dots,j_r\in \{1,\dots,n\}$ so that
$$\bp\in  [e_{i_1}^p,\dots,e_{i_r}^p]\subset \Sigma_p\mbox{ and }
\bq\in  [e_{j_1}^q,\dots,e_{j_r}^q] \subset \Sigma_q$$
and so that moreover the $i_1,i_2,\dots,i_r$ components of $\bp$ and 
the $j_1,j_2,\dots,j_r$ coordinates of $\bq$ 
are non-zero. So $(\bp,\bq)$ is a completely  mixed Nash equilibrium w.r.t.
to this $r\times r$ submatrix
and  
$$\BR_p(\bq)=[e^p_{i_1},\dots,e^p_{i_r}]\mbox{ and }
\BR_q(\bp)=[e^q_{j_1},\dots,e^q_{j_r}].$$
\item The Nash equilibrium $(\bp,\bq)$ of the game
is unique.
\end{enumerate}
\end{prop}

\begin{proof} If  $(p,q)\in \Sigma_p\times \Sigma_q$ is so that
$\BR_p(q)$ and $\BR_q(p)$ are both  singletons,
then one component of $\M q$ is strictly larger than the others, 
and one component of $p \M$ is strictly smaller than the others.
Hence there exists a neighbourhood $U$ of $(p,q)$ and $i \in \{1,2,\dots,m\}, j\in \{1,\dots,n\}$ so that $\BR_p(q)=e_i^p$ and $\BR_q(p)=e_j^q$
for all $(p,q)\in U$. Because of (\ref{eq:rewriteH}) the function $H$ is then affine on this neighbourhood $U$. 
As we will show next, the transversality 
assumption (\ref{eq:trans2}) implies that the set of $q\in \Sigma$, for which $q\mapsto \BR_p(q)$ is multi-valued,
is a codimension-one hyperplane in $\Sigma$ (and similarly for $p\mapsto \BR_q(p)$). 

Consider a linear space $Z_p\subset {\mathbb R}^m$
of $p\in \R^m$ where $\BR_q(p)=[e^q_{j_1},\dots,e^q_{j_k}]$
where $j_1,j_2,\dots,j_k\in \{1,\dots,n\}$.
This space is equal to the set of $p\in \R^m$ where $p' \MS=0$,
where $\MS$ is the $(k-1)\times n$ matrix made up from 
the difference of the $j_2,\dots,j_k$-th and the $j_1$-th column vector of the matrix $\M$. 
By  considering the matrix $\MB$ which consists
of only the $i_1,\dots,i_r$-th rows of $\MS$, we get the intersection of $Z_p$  with 
$[e_{i_1}^p,\dots,e_{i_r}^p]$. Since, by assumption (\ref{eq:trans2})
the matrix $\MS$ has maximal rank, it follows that this intersection
has codimension $(k-1)$ in $[e_{i_1}^p,\dots,e_{i_r}^p]$.
So if $r<k-1$ then the intersection is empty. The intersection will be also empty when the
intersection of $Z_p$ with the linear space spanned by $[e_{i_1}^p,\dots,e_{i_r}^p]$
is outside  $\Sigma_p$.

Let $(\bp,\bq)$ be a Nash equilibrium and assume $\BR_q(\bp)$ has dimension 
$k$  and $\BR_p(\bq)$ has dimension $l$.  To be definite assume that $k\le l$.
Let $\BR_q(\bp)=[e^q_{j_1},\dots,e^q_{j_k}]$ 
where $j_1,\dots,j_k \in \{1,\dots,n\}$. Note that by the first part of the lemma we have that $k\le m$. 
Moreover,  $\bq\in \BR_q(\bp)=[e^q_{j_1},\dots,e^q_{j_k}]$.
But using the first part of the lemma again it follows that
$\BR_p(\bq)$ can have at most dimension $k$ (and dimension $<k$ 
if one of the $j_1,\dots,j_k$ components of $\bq$ is zero). 
Since we assumed that $k\le l$ we get that $k=l$ and that the 
$j_1,\dots,j_k$-th components of  $\bq$ are strictly positive (and the others zero). 
Similarly, the $j_1,\dots,j_k$-th components of  $\bp$  are strictly positive (and the
others zero).

Using Proposition~\ref{prop:unique1} the Nash equilibrium $(\bp,\bq)$ is unique within this $k\times k$
sub matrix. If the game has another Nash equilibrium, then a convex combination
of these is also a Nash equilibrium and so there would be a completely mixed Nash equilibrium
in a $k'\times k'$ subgame with $k'> k$ (containing the $k\times k$ subgame from before).
But this contradicts the uniqueness of completely mixed Nash equilibria we already established.
\end{proof}

\section{An associated Hamiltonian vector field}

Let us associate a Hamiltonian vector field $X_H$  to $H$ when $m=n$.
We will denote by $T\Sigma$ the tangent plane to $\Sigma$, i.e.
the set of vectors in $\R^n$ whose components sum up to zero. 
As we will see in the next section, even though $X_H$ is not continuous and 
multivalued, it does have a continuous flow.

\begin{theorem}
\label{thmA}
Assume that $\M$ is an $n\times n$ matrix that satisfies the  transversality assumption
 (\ref{eq:trans2}) and let $H\colon \Sigma\times \Sigma\to \R$
be as in (\ref{eq:hamilton}).  Moreover,  take the symplectic two-form
$\omega=\sum_{ij} a_{ij} dp_i \wedge dq_j$ and let $A$ be the $n\times n$ matrix
$A=(a_{ij})$.
Then the following properties hold.
\begin{enumerate}
\item Let $X_H$ be the vector field 
associated to  $H\colon \Sigma\times \Sigma\to \R$ and the symplectic form
$\omega$, i.e. $X_H$ is the vector field
tangent to level sets of $H\colon \Sigma
\times \Sigma \to \R$ so that $\omega(X_H,Y)=dH(Y)$ for all vector fields $Y$
which are tangent to $\Sigma\times \Sigma$. Then $X_H$ is multivalued,
and the differential inclusion
corresponding to $X_H$ is:
\begin{equation}
\begin{array}{rl}
\dfrac{dp}{dt}  & \in   \,\,  P_p\, A'^{-1} \,\, \M'\, \BR_p(q),\\ 
\dfrac{dq}{dt}  & \in  \,\, P_q\, A^{-1} \, \M \,\, \BR_q(p).
\end{array} 
\label{eq:hamiltonvf}
\end{equation}
where $P_p,P_q\colon \R^n\to   T\Sigma$ is the projection along 
$A'^{-1} \ii$ respectively $A^{-1}\ii$. 
\item \label{itemcommute}
\begin{equation}(P_q A'^{-1})'=
P_p A^{-1}\label{eq:commute}
\end{equation}
\item On each level set of $H^{-1}(\rho)$, $\rho>0$, the flow is piecewise a translation flow, and has no 
stationary point.
\item First return maps between certain hyperplanes are piecewise affine maps.
\end{enumerate}
\end{theorem}
\begin{proof}
Let us compute the Hamiltonian vector field
$X_H$  corresponding to a symplectic two-form $\omega=\sum_{ij}a_{ij} dp_i \wedge dq_j$,  where we assume that $a_{ij}$ are constants so that the matrix  $A=(a_{ij})$ is invertible.  Notice that
we consider $H$ as a function on  $\Sigma\times \Sigma$ (rather
than on $\R^n\times \R^n$). So, by definition,
$X_H$ is the vector field which is tangent to the hyper plane  $\Sigma$
and for which 
$\omega(X_H,Y)=dH\cdot Y$ for every vector field $Y$ which is tangent
to $\Sigma$. Write
$X_H=\sum_j b_j \dfrac{\partial}{\partial p_j}+\sum_i c_i \dfrac{\partial}{\partial q_i}$ and $Y=\sum_j u_j \dfrac{\partial}{\partial p_j}+\sum_i v_i \dfrac{\partial}{\partial q_i}$
and for simplicity let $b,c,u,v$ be the  column matrices corresponding to the coefficients $b_j,c_i,u_j,v_i$. That $X_H$ and $Y$ are tangent to $\Sigma\times \Sigma$ means that the sum of the elements of $b,c,u,v$ are equal to $0$.
Notice that  $\omega(X_H,Y)=b' A v- u' A c$
and 
$dH\cdot Y=(\dfrac{\partial H}{\partial p_1},\dots, \dfrac{\partial H}{\partial p_n}) u +  (\dfrac{\partial H}{\partial q_1},\dots, \dfrac{\partial H}{\partial q_n})  v$.
Hence $v' A' b=v'(\dfrac{\partial H}{\partial q_1},\dots, \dfrac{\partial H}{\partial q_n})'$ and $-c' A' u=(\dfrac{\partial H}{\partial p_1},\dots, \dfrac{\partial H}{\partial p_n}) u $ for all $u,v\in \R^n$ with sum zero.
It follows that $A'b=(\dfrac{\partial H}{\partial q_1},\dots, \dfrac{\partial H}{\partial q_n})'+\lambda_1 \underline 1$ and 
$c' A' =-(\dfrac{\partial H}{\partial p_1},\dots, \dfrac{\partial H}{\partial p_n}) + \lambda_2 \underline 1'$. 
So  $b=A'^{-1}(\dfrac{\partial H}{\partial q_1},\dots, \dfrac{\partial H}{\partial q_n})'+\lambda_1 A'^{-1} \underline 1$
and $c =-A^{-1}(\dfrac{\partial H}{\partial p_1},\dots, \dfrac{\partial H}{\partial p_n})' + \lambda_2 A^{-1}\underline 1$. Here $\lambda_1, \lambda_2$ are so that
the sum of the elements of $c$ and $d$ is zero.
To compute $\dfrac{\partial H}{\partial p_i}$ and $\dfrac{\partial H}{\partial q_i}$,
notice that 
 $q\mapsto \BR_p(q)$ and $p\mapsto \BR_q(p)$ are piecewise constant
 and rewrite  
 $H(p,q)=\M' \, \BR_p(q) \,  \cdot q- \M \, \BR_q(p) \cdot p$
in order to make the role of $p,q$ more symmetric, where $\cdot$ stands
for the inner product and $\M'$ for the transpose of $\M$. It follows that 
$(\dfrac{\partial H}{\partial q_1},\dots, \dfrac{\partial H}{\partial q_n})'=
\M' \, \BR_p(q)$ and 
$(\dfrac{\partial H}{\partial p_1},\dots, \dfrac{\partial H}{\partial p_n})'
= -\M \, \BR_q(p)$.
Hence
$$b=A'^{-1} \M' \, \BR_p(q) + \lambda_1 A'^{-1} \ii$$ and 
$$c =A^{-1} \M \, \BR_q(p) + \lambda_2   A^{-1}\ii.$$
where, as mentioned, $\lambda_i$ are chosen so that 
the sum of the elements of $b$ and $c$ is zero.
In other words, the differential inclusion associated to the vector field
$X_H$ becomes 
\begin{equation*}
\begin{array}{rl}
\dfrac{dp}{dt}  & \in  \,\,  P_p\, A'^{-1} \,\, \M'\, \BR_p(q),\\ 
\dfrac{dq}{dt}  & \in   \,\, P_q\, A^{-1} \, \M \,\, \BR_q(p).
\end{array} 
\end{equation*}
where $P_p,P_q\colon \R^n\to   T\Sigma$ is the projection along 
$A'^{-1} \ii$ respectively $A^{-1}\ii$. 

\renewcommand\a{\hat a}
In order to prove (\ref{eq:commute}), note
$$
\begin{array}{rll}
P_p A'^{-1}z &\in A'^{-1}z-\dfrac{1}{A'^{-1}\ii \cdot \ii}(A'^{-1}z \cdot \ii)  \cdot A'^{-1}\ii \\
P_q A^{-1}z &\in A^{-1}z-\dfrac{1}{A^{-1}\ii \cdot \ii}(A^{-1}z \cdot \ii)  \,\,  \cdot A^{-1}\ii
\end{array}
$$
We also note that $A'^{-1}\ii \cdot \ii= \ii' A'^{-1} \ii=
\ii' A^{-1} \ii=A^{-1}\ii \cdot \ii$.
Denote the coefficients of the matrix $A^{-1}$ by $\a_{ij}$.  Then 
$$A^{-1}z \cdot  \ii = (\a_{11}+\dots +\a_{n1})z_1 +  (\a_{12}+\dots +\a_{n2})z_2 + \dots +  (\a_{1n}+\dots +\a_{nn})z_n$$
and 
$$A^{-1}\underline 1=\left( 
\begin{array}{c}\sum_j \a_{1j} \\ \sum_j \a_{2j}\\ \vdots \\  \sum_j \a_{nj} \end{array}\right).$$
Hence
$(A^{-1}z \cdot  \ii) \, A^{-1} \ii$  is equal to 
$$\left( 
\begin{array}{cccc}
(\sum_j \a_{1j})(\sum_i \a_{i1}) & (\sum_j \a_{1j})(\sum_i \a_{i2})  & \dots & 
(\sum_j \a_{1j})(\sum_i \a_{in}) 
\\
(\sum_j \a_{2j})(\sum_i \a_{i1}) & (\sum_j \a_{2j})(\sum_i \a_{i2})  & \dots & 
(\sum_j \a_{2j})(\sum_i \a_{in}) \\
\vdots & \vdots & \ddots  &  \vdots \\
(\sum_j \a_{nj})(\sum_i \a_{i1}) & (\sum_j \a_{nj})(\sum_i \a_{i2})  & \dots & 
(\sum_j \a_{nj})(\sum_i \a_{in})
\end{array}\right) z.$$
From this (\ref{eq:commute}) follows. 
\end{proof}

\begin{lemma}\label{lem:AM}
Suppose that $A=\M$ then
\begin{equation*}
\begin{array}{rl}
\dfrac{dp}{dt}  &\in  \,\,  \BR_p(q) - \bar p,\\ 
\dfrac{dq}{dt}  &\in  \,\,  \BR_q(p) - \bar q.
\end{array} 
\end{equation*}
\end{lemma}
\begin{proof} Notice that $\M\bar q\in <\underline 1'>$ and 
$\bar p' \M\in <\ii'>$. It follows that if we take $A=\M$
then $A'^{-1} \ii=\M'^{-1}\ii= (\ii' \M^{-1})'=\bar p$
and $A^{-1}\ii=\M^{-1}\ii=\bar q$. 
So  (\ref{eq:hamiltonvf}) becomes 
\begin{equation*}
\begin{array}{rl}
\dfrac{dp}{dt}  &\in  \,\,  P_p\, \, \BR_p(q),\\ 
\dfrac{dq}{dt}  &\in  \,\, P_q\,  \,\, \BR_q(p),
\end{array} 
\end{equation*}
where $\hat P_p,P_q\colon \R^n\to T \Sigma$ is the projection along 
$\bar p=A'^{-1} \ii$ respectively $\bar q=A^{-1}\ii$. 
The lemma  follows. 
\end{proof}

\section{Continuity of flows of related differential inclusions}
In the next theorem we consider related differential inclusions 
and show that these have a unique flow associated to them and that
this flow continuous.  Then, in the next section, we will show
that this implies continuity of the flow associated to the Hamiltonian
vector fields.

That the flow is continuous is not entirely obvious,
in particular because orbits can lie entirely in the set where both $q\mapsto \BR_p(q)$ and $p\mapsto \BR_q(p)$
are multi-valued.

In this section we do {\em not} need to assume $m=n$.
Let $\hat \Sigma_p,\hat \Sigma_q$ be the hyperplanes containing $\Sigma_p$
resp. $\Sigma_q$. Let $\M$ be as before a $m\times n$ matrix
with a completely mixed Nash equilibrium $(\bp,\bq)$.
Let $\alpha\in \R$ and let $X$ be a $m\times m$ matrix and $Y$ a $n\times n$ matrix.
Assume that 
\begin{equation}
\M Y=X'\M\label{eq:commutexy}
\end{equation}
\begin{equation}X(p)+\alpha \bp \in \hat \Sigma_p\,\,\mbox { for all }\,p\in  \Sigma_p,\label{eq:level1}\end{equation}
\begin{equation}Y(q)+\alpha\bq \in \hat \Sigma_q\,\,\mbox { for all }\,q\in  \Sigma_q.\label{eq:level2}\end{equation}
Furthermore, 
consider the following  differential inclusion
\begin{equation}
\dfrac{dp}{dt}\in X\BR_p(q)+\alpha \bp -p, \dfrac{dq}{dt}\in Y\BR_q(p)+\alpha \bq -q.
\label{eq:genBR}
\end{equation}

\begin{lemma}\label{lem:contatNE}
Let $\M$ be as before a $m\times n$ matrix
with a completely mixed Nash equilibrium $(\bp,\bq)$.
Consider
$H(p,q)=(\BR_p(q))' \M  q- p' \M  \BR_q(p)$
and let $\M,X,Y$ be as in (\ref{eq:commutexy})-(\ref{eq:level2})
and consider the differential inclusion (\ref{eq:genBR}) . 
Then 
\begin{enumerate}
\item the above differential inclusion has solutions;
\item  solutions $t\mapsto (p(t),q(t))$  of (\ref{eq:genBR}) stay in $\hat \Sigma_p\times \hat \Sigma_q$ (but in principle components could become negative);
\item
$\dfrac{dH}{dt}=-H$ and 
so $H$ tends to zero along orbits as $H(t)=ce^{-t}$.
\end{enumerate}
Moreover, if the
Nash equilibrium is unique, the motion of the differential inclusion  (\ref{eq:genBR})
is continuous at this  Nash equilibrium
and orbits do {\em not} reach the Nash equilibrium in finite time. 
\end{lemma}
Remark that at this point we do not yet know that there exists a unique solution
of the differential inclusion. The above lemma shows that any flow
is continuous at the Nash equilibrium. In the Theorem~\ref{thm:trans}
we show that the flow is unique and continuous everywhere.

\begin{proofof}{Lemma~\ref{lem:contatNE}}
 That the differential inclusion (\ref{eq:genBR}) has solutions
holds because $p\mapsto \BR_q(p)$ and $q\mapsto \BR_p(q)$ are upper semi-continuous, see  \cite{AubinCellina84}.
Note that (\ref{eq:level1}) and (\ref{eq:level2}) imply that the solutions of
the differential inclusion remain in $\Sigma$. 
$H$ is continuous by definition
of $\BR_p$ and $\BR_q$. 
To show that solutions of (\ref{eq:genBR})  tend to Nash equilibria, we show that 
$H(p,q)=\BR_p(q)' \M  q- p' \M  \BR_q(p)$
is a Lyapunov function.

As we saw, $H\ge 0$ and $H(p,q)=0$ iff $(p,q)$ is a Nash equilibrium. 
Since $\BR_p$ and $\BR_q$ are piecewise constant,
(\ref{eq:genBR}) implies 
$$\begin{array}{rl}
\dfrac{dH}{dt} &=\BR_p(q)' \M \dfrac{dq}{dt} - \dfrac{dp}{dt}' \M  \BR_q(p)\\
\\
& =\BR_p(q)' \M ( Y\BR_q(p)+\alpha \bq -q) - (X\BR_p(q)+\alpha \bp-p)'\M  \BR_q(p)=-H\end{array}$$ 
where we used that $\M Y= X'\M$ and $p'\M \bq=\bp'\M \bq = \bp' \M q$ for all $q\in \Sigma_p$ and $q\in \Sigma_q$. 
\end{proofof}

\begin{remark}
If $m=n$ and the transversality condition
(\ref{eq:trans1}) 
then properties (\ref{eq:commutexy}), (\ref{eq:level1}) and (\ref{eq:level2}) imply that $X\bp = (1-\alpha) \bp$
and $Y\bq=(1-\alpha)\bq$. Indeed, $Yq+\alpha \bq\in \Sigma_q$ implies that
$\ii' Y q+\alpha \ii' \bq =1$ and so $\ii' Y q=1-\alpha$. Moreover, $\bp'\M=\lambda \ii'$  for 
some $\lambda$.  Since $\M Y=X'\M$ this implies $(1-\alpha)\lambda=\lambda \ii 'Y q=
\bp' \M Y q= \bp' X' \M q$ for all $q\in \Sigma_q$.
Hence   $\bp' X'  \M$ is a multiple of $\ii'$. By the transversality condition $\M$ is invertible 
and since $\bp' \M=\ii'$
it follows that $\bp' X'$ is a multiple of $\bp'$.  Therefore,  by (\ref{eq:level1}), 
$X\bp=(1-\alpha)\bp$.
\end{remark}

\begin{definition}[Indifference sets]
Assume $1\le j_1<\dots<j_k\le n$ and $1\le i_1<\dots<i_l\le m$.
Then  $Z_p(j_1,\dots,j_k)\subset \Sigma_p$ is defined to be the set of
of $p\in \Sigma_p$ where $\BR_q(p)=[e^q_{j_1},\dots,e^q_{j_k} ]$.
Similarly, $Z_q(i_1,\dots,i_l)$  is defined to be the set of
$q\in \Sigma_q$ where $\BR_p(q)=[e^p_{i_1},e^p_{i_2},\dots,e^p_{i_l}]$.
\end{definition}

\begin{assumption}[Third transversality assumption on $X,Y$]\label{trans3}
Let $\bp,\bq$ be a completely mixed Nash equilibrium w.r.t. $\M$.
Furthermore, let $\alpha\in \R$ and let $X,Y$ be as in (\ref{eq:level1}) and (\ref{eq:level2}). 
For each $1\le j_1<\dots<j_k\le n$ and $1\le  i_1<\dots<i_l\le m$ we require that 
the subspaces  associated to $Z_p(j_1,\dots,j_k)$ and $[X(e_{i_1})+\alpha \bp,\dots,X(e_{i_l})+\alpha \bp]$ 
are in general position w.r.t. $\Sigma_p$. Similarly, we require that 
$Z_q(i_1,\dots,i_l)$ and $[Y(e_{j_1})+\alpha \bq,\dots,Y(e_{j_k})+\alpha\bq]\cap \Sigma_q$
are in general position w.r.t. $\Sigma_q$.
\end{assumption}

Note that if $X$ and $Y$ are the identity matrices, this assumption follows from the 2nd transversality assumption,
see the 2nd item in Proposition~\ref{prop:unique2}. Before we can prove the main 
theorem of this section, we need one more lemma.
In game theory, this restriction to a subspace
would be referred to as the {\lq}dynamics associated to a subgame{\rq}.

\begin{lemma}[Restriction of dynamics]
Assume that the previous transversality assumption holds and that
$Z_p(j_1,\dots,j_k)\times Z_q(i_1,\dots,i_l)$ and
$$[X(e_{i_1}^p)+\alpha \bp,\dots,X(e_{i_l}^p)+\alpha \bp]\times [Y(e_{j_1}^q)+\alpha \bq,\dots,Y(e_{j_k}^q)+\alpha \bq]
$$
intersect in a unique point $(a,b)$. Then
define 
$$\BR_p^{i_1,\dots,i_l}\colon <Y(e_{j_1}^q)+\alpha \bq,\dots,Y(e_{j_k}^q)+\alpha \bq>\quad \to \quad <e^p_{i_1},\dots,e_{i_l}^p> $$ and 
$$\BR_q^{j_1,\dots,j_k}\colon <X(e_{i_1}^p)+\alpha \bp,\dots,X(e_{i_l}^p)+\alpha \bp>\quad \to \quad <e^q_{j_1},\dots,e^q_{j_k}>$$
by 
\begin{equation}
\BR_p^{i_1,\dots,i_l}(q):=\argmax_{p\in <e^p_{i_1},\dots,e_{i_l}^p>}  p'\M q \,\, 
\mbox{ and }\,\,  \BR_q^{j_1,\dots,j_k}(p):=\argmin_{q\in <e^q_{j_1},\dots,e^q_{j_k}>} \, p' \M q .
\label{eq:brpqres}
\end{equation}
Then 
\begin{itemize}
\item $(a,b)$ is a completely mixed Nash
equilibrium in the sense that
all components of $a$ and $b$ are strictly positive and 
that $\BR_p^{i_1,\dots,i_l}(q)=<e^p_{i_1},\dots,e_{i_l}^p>$
and $\BR_q^{j_1,\dots,j_k}(p):=<e^q_{j_1},\dots,e^q_{j_k}>$.
\item
Moreover, all orbits of 
\begin{equation}
\dfrac{dp}{dt}\in X\BR_p^{i_1,\dots,i_l}(q)+\alpha \bp -p, \dfrac{dq}{dt}\in Y\BR_q^{j_1,\dots,j_k}(p)+\alpha \bq -q.
\label{eq:genBRres}
\end{equation} remain
in the space $[X(e_{i_1}^p)+\alpha \bp,\dots,X(e_{i_l}^p)+\alpha \bp]\times [Y(e_{j_1}^q)+\alpha \bq,\dots,Y(e_{j_k}^q)+\alpha \bq]$ and converge to $(a,b)$. The motion of the differential inclusion is continuous {\em at} $(a,b)$. 
\end{itemize}
\end{lemma}
\begin{proof}
That all components of $a$ and $b$ are strictly positive follows from the transversality condition. 
That orbits remain in this space follows from the definition, and the previous lemma
implies that orbits converge to $(a,b)$. 
\end{proof}

Now we will prove that, provided transversality conditions hold,
the flow is unique and continuous.

\begin{theorem} [Transversality condition implies continuity and uniqueness] \label{thm:trans}
Let $\M$ be as before a $m\times n$ matrix
with a completely mixed Nash equilibrium $(\bp,\bq)$.
Let $\alpha\in \R$ and let $X$ be a $m\times m$ matrix and $Y$ a $n\times n$ matrix
as in (\ref{eq:commutexy})-(\ref{eq:level2})
and consider the differential inclusion (\ref{eq:genBR}) . 
Moreover, assume the transversality assumption (\ref{eq:trans1}) and (\ref{trans3}) hold.
Then \begin{enumerate}
\item Through each point $(p,q)\ne (\bp,\bq)$ there exists a unique solution
(the orbit is unique for all time or at least until such time that it hits  $(\bp,\bq)$). 
\item  the motion defined by equations (\ref{eq:genBR})  forms a continuous flow.
\end{enumerate}
\end{theorem}
\begin{proof}
Let us first show that the above condition is enough to 
get that orbits move transversally (and with positive speed) 
through  the sets in which  either $\BR_q(p)$ or $\BR_p(q)$ is multivalued, but  {\em not} both. 
In order to be definite, assume
that $\BR_q(p)$ is multivalued, say $\BR_q(p)\subset [e^q_l,e^q_{l'}]$, and that $\BR_p(q)=\{e^p_i\}$. 
Note that the set of  $p\in \Sigma_p$ for which 
$\BR_q(p)\supset [e^q_l,e^q_{l'}]$ is contained in the
set $Z_p(l,l')=\{p\in \Sigma_p; \BR_q(p)=[e^q_l,e^q_{l'}]\}$.
By assumption the transversality assumption $X(e_i)+\alpha+\bp$ is not contained
in $Z_p(l,l')$ and so by the form of the vector field (\ref{eq:genBR}), orbits 
move off the indifference space $Z_{l,l'}$ with positive speed.



Let us next consider the situation that both 
$\BR_q(p)$ or $\BR_p(q)$ are multivalued. To analyze this situation, consider $(p,q)$ where 
$\BR_q(p)=[e^q_{j_1},e^q_{j_2},\dots,e^q_{j_k}]$ and $\BR_p(q)=[e^p_{i_1},e^p_{i_2},\dots,e^p_{i_l}]$.
So $p\in Z_p(j_1,\dots,j_k)$ and
$q\in Z_q(i_1,\dots,i_l)$. Let us abbreviate these spaces as $Z_p$ and $Z_q$.  Starting from $(p,q)$,
the $p$ component can move (by the form of the vector field)
towards any convex combination of $X(e_{i_1}^p)+\alpha \bp,\dots,X(e_{i_l}^p)+\alpha \bp$.
Hence the dimension of the space of directions in which $p$ can move (starting from
$(p,q)$) while staying in $Z_p$ is equal to the dimension $d$ of $Z_p(j_1,\dots,j_k)\cap 
[X(e_{i_1}^p)+\alpha \bp,\dots,X(e_{i_l}^p)+\alpha \bp]$. 
By the the above transversality assumption  $m-1-d=\codim(Z_p\cap [X(e_{i_1}^p)+\alpha \bp,\dots,X(e_{i_l}^p)+\alpha \bp])$
is equal to 
$$
\codim(Z_p)+\codim( [X(e_{i_1}^p)+\alpha \bp,\dots,X(e_{i_l}^p)+\alpha \bp])=(k-1)+(m-1-(l-1))=(m-1)+(k-l).$$
In particular, $d=l-k$ and so 
if $l<k$ then $d<0$ and the intersection is empty:  the vector field moves $p$ immediately (i.e. with 
positive speed and transversally) 
off $Z_p$.  So to stay inside $Z_p$ we need $l\ge k$. (Note that even if $l\ge k$ the intersection can be empty because
$[X(e_{i_1}^p)+\alpha \bp,\dots,X(e_{i_l}^p)+\alpha \bp]$ is not a linear space, but only a simplex within such a linear space. If the intersection is empty, then the orbit still moves immediately off $Z_p$.)
Similarly interchanging the role of $p$ and $q$,
the dimension of the space of directions $q$ can move while staying in $Z_q$  is at most $k-l$.
It follows that only when $k=l$ it is possible for the orbit through 
$(p,q)$ to stay within $Z_p\times Z_q$.
In other words, then 
$$Z_p\cap 
[X(e_{i_1}^p)+\alpha \bp,\dots,X(e_{i_l}^p)+\alpha \bp]\mbox{ and } 
Z_q\cap 
[Y(e_{j_1}^q)+\alpha \bq,\dots,Y(e_{j_k}^q)+\alpha \bq]$$
both consist of one point, say $(a,b)$.
In other words,  
\begin{equation}\begin{array}{rl}
&Z_p\times Z_q\mbox{ and }
[X(e_{i_1}^p)+\alpha \bp,\dots,X(e_{i_l}^p)+\alpha \bp]\times [Y(e_{j_1}^q)+\alpha \bq,\dots,Y(e_{j_k}^q)+\alpha \bq]\\
 \\
&\mbox{intersect in a unique point $(a,b)$ and have complimentary dimensions}
\\
&\mbox{(as subsets of $\Sigma_p\times \Sigma_q$).}
\end{array}
\label{eq:compl}
\end{equation}
So there is at most one orbit starting from $(p,q)$ 
which stays within $Z_p\times  Z_q$ (the orbit moving towards $(a,b)$). 
Note that even if there exists an orbit through $(p,q)$ which stays within $Z_p\times Z_q$
there may be another solution which moves off this set. We will show that an 
orbit which starts slightly off the set $Z_p\times Z_q$ spirals around it, as in the middle diagram
in Figure~\ref{fig:coiling}.
Here we use the previous lemma. In this way, it will follow that the flow becomes continuous and unique.

So let prove continuity at a point $(p_0,q_0)$ in $Z_p\times Z_q$
where we use the notation from before. 
Notice that we have shown that only when $k=l$ an orbit through $(p_0,q_0)$
does not necessarily move transversally through $Z_p\times Z_q$.
Near $(p_0,q_0)$  only strategies $j_1,\dots,j_k$
are optimal for  $q$ and strategies $i_1,\dots,i_k$
are optimal for $p$.  
By (\ref{eq:compl}) each $(p,q)$ near $(p_0,q_0)$
can be uniquely decomposed as 
$$(p,q)=(\tilde p,\tilde q)\oplus (\hat p,\hat q),$$
where $(\tilde p,\tilde q)\in (Z_p-a)\times (Z_q-b)\subset T\Sigma_p\times T\Sigma_q$
and $(\hat p,\hat q)\in [X(e_{i_1}^p)+\alpha \bp,\dots,X(e_{i_l}^p)+\alpha \bp]\times [Y(e_{j_1}^q)+\alpha \bq,\dots,Y(e_{j_k}^q)+\alpha \bq]$. Note that in this unique decomposition, \begin{equation}
(p,q)=(\tilde p,\tilde q)\oplus (a,b) \mbox{ for any } (p,q)\in Z_p\times Z_q.\label{split}
\end{equation}
Rewrite (\ref{eq:genBR}) as 
\begin{equation}
\dfrac{d\tilde p}{dt}+\dfrac{d\hat p}{dt}\in -\tilde p + (X\BR_p(q)+\alpha \bp -\hat p),\quad  \dfrac{d\tilde q}{dt}+\dfrac{d\hat q}{dt}\in -\tilde q + (Y\BR_q(p)+\alpha \bq -\hat q).
\label{eq:genBR2}
\end{equation}
Note that $\BR_p(q)=\BR_p^{i_1,\dots,i_k}(\hat q)$ and $\BR_q(p)=\BR_q^{j_1,\dots,j_k}(\hat p)$,
where we use the notation from the previous lemma. 
Hence an orbit $t\mapsto (p(t),q(t))$ through $(p(0),q(0))=(p,q)$ of the differential inclusion is of the form
\begin{equation}
p(t)=\tilde p(t)+\hat p(t),\quad \quad  \quad\, q(t)=\tilde q(t)+\hat q(t)
\end{equation}
\begin{equation}
\dfrac{d\tilde p}{dt}=-\tilde p(t),\quad \quad  \quad\, \dfrac{d\hat q}{dt}=-\tilde q(t)
\end{equation}
\begin{equation}
 \dfrac{d\hat p}{dt}\in X\BR_p^{i_1,\dots,i_k}(\hat q)+\alpha \bp -\hat p,\quad \quad \quad
\dfrac{d\hat q}{dt} \in Y\BR_q^{j_1,\dots,j_k}(\hat p)+\alpha \bq -\hat q
\label{eq:genBRred}
\end{equation}
Since $(\tilde p(0),\tilde q(0))+(a,b)$ and  $(a,b)$ are both $Z_p\times Z_q$
we have that $(\tilde p(t),\tilde q(t))+(a,b)=e^{-t}(\tilde p(0),\tilde q(0))+(a,b)\in
Z_p\times Z_q$ and therefore 
$(\tilde p(t),\tilde q(t))\in (Z_p-a)\times (Z_q-b)$. Moreover,  from the previous lemma,
$(\hat p(t),\hat q(t))\in [X(e_{i_1}^p)+\alpha \bp,\dots,X(e_{i_l}^p)+\alpha \bp]\times [Y(e_{j_1}^q)+\alpha \bq,\dots,Y(e_{j_k}^q)+\alpha \bq]$.
So $(p(t),q(t))$ splits as before
as $(p(t),q(t))=(\tilde p(t),\tilde q(t))\oplus (\hat p(t),\hat q(t))$ and 
the above equations are equivalent to 
\begin{equation}
p(t)=e^{-t} \tilde p(0)+\hat p(t),\quad \quad  \quad\, q(t)=e^{-t} \tilde q(0)+\hat q(t).
\label{eq:tilde1}\end{equation}
and 
\begin{equation}
 \dfrac{d\hat p}{dt}\in X\BR_p^{i_1,\dots,i_k}(\hat q)+\alpha \bp -\hat p,\quad \quad \quad
\dfrac{d\hat q}{dt} \in Y\BR_q^{j_1,\dots,j_k}(\hat p)+\alpha \bq -\hat q.
\label{eq:genBRred3}
\end{equation}
By the previous lemma, orbits under (\ref{eq:genBRred3}) converge to $(a,b)$. 
From (\ref{split}) it follows that the flow under (\ref{eq:tilde1}) and (\ref{eq:genBRred3}) are unique and continuous at any 
$(p,q)\in Z_p\times Z_q$. Since the flow is continuous and unique elsewhere,
the flow is globally unique and continuous. 
\end{proof}

\begin{figure}[htp]
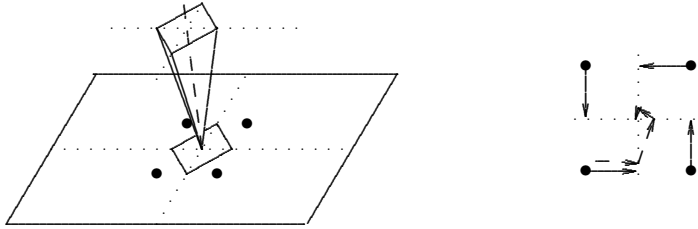
 \hfil
\beginpicture
\dimen0=0.2cm \dimen1=0.07cm
\setcoordinatesystem units <\dimen0,\dimen0>
\setplotarea x from 0 to 25, y from -2 to 13
\plot 0 0 20 0 26 10 6 10 0 0 /

\setdots 
\plot 3 5 23 5 /
\plot 10 0 16 10 /
\setsolid
\plot 13 5 11 11.33 14 13 13 5 10 13 11 11.33 /
\plot 10 13 13 14.66 14 13 /
\setdots
\plot 7 13 20 13 /
\plot 10 9.66 14 16.33 /
\plot 12 3.33 15 5 14 6.66 11 5 12 3.33 /
\setsolid
\plot 12 3.33 15 5 14 6.66 13.2 6.22 /
\plot 12.5 5.82 11 5 12 3.33 /
\multiput {$\bullet$} at 10 3.33 14 3.33 12 6.66 16 6.66 /
\setdashes
\plot 13 5 12 13 11.8 14.6 11.7 15.4 /
\setcoordinatesystem units <\dimen1,\dimen1> point at -120 -20
\setplotarea x from -10 to 10, y from -10 to 10
\setlinear \setdots
\plot -12 0 12 0 /
\plot 0 12 0 -12 / \setsolid
\put {$\bullet$} at 10 10
\put {$\bullet$} at -10 10
\put {$\bullet$} at -10 -10
\put {$\bullet$} at 10 -10
\arrow <5pt> [0.2,0.4] from 10 10 to 0.5 10
\arrow <5pt> [0.2,0.4] from -10 10 to -10 0.5
\arrow <5pt> [0.2,0.4] from -10 -10 to -0.5 -10
\arrow <5pt> [0.2,0.4] from 10 -10 to 10 -0.5
\setdashes
\arrow <5pt> [0.2,0.4] from -8 -8 to 0 -8.5
\arrow <5pt> [0.2,0.4] from 0 -8.5 to 3 0
\arrow <5pt> [0.2,0.4] from 3 0 to 0 2
\arrow <5pt> [0.2,0.4] from 0 2 to -0.5 0
\endpicture
\caption{\label{fig:quad}
{\small This picture illustrates what happens near $Z_p\times Z_q$. The dashed line denotes this set.
The targets are drawn marked as $\bullet$,
and by choosing appropriate combinations of the targets orbits can move along the dashed line.
Nearby orbits spiral along some cone towards the apex of the cone, always
aiming for one of the points marked as $\bullet$. Note that the vector field
on the cone is {\bf not} close to the vector field along the dashed line.
The reason orbits move like this is that in restriction
to the target plane, the dynamics has this property.}}
\end{figure}

\begin{remark}
This use of $H$ as a Lyapunov function for the differential inclusion
$\dot p=\BR_p(q)-p, \dot q=\BR_q(p)-q$
goes back to Brown \cite{Brown49} and \cite{Brown51} and was 
explicitly mentioned by Hofbauer in \cite{Hofbauer95}. 
 Harris \cite{Harris98} used a related method to analyse the speed of convergence to the value of the game. 
\end{remark}
 
\section{Projecting the flow on level sets of $H$}
Consider the map $\pi\colon \Sigma\times \Sigma \setminus (\bp,\bq) \to H^{-1}(\rho)$
from (\ref{eq:pi}).  In this section we will show that
orbits of (\ref{eq:genBR}) project to orbits on the $(n+m-3)$-dimensional topological sphere 
$H^{-1}(\rho)$. 

\begin{prop}\label{prop:projection}
The map $\pi$ projects orbits of (\ref{eq:genBR}) onto
orbits  of 
\begin{equation}
\dfrac{d\tilde p}{ds}=X\BR_p(\tilde q)\,+(\alpha-1)\bp \, , \,\, \dfrac{d\tilde q}{ds}=Y\BR_q(\tilde p)+(\alpha-1)\bq.
\label{eq:inducedflow}
\end{equation}
Here we use the reparametrization $s=e^t$.
\end{prop}

\begin{remark}
There are a few special cases that deserve attention:
\begin{enumerate}
\item Consider the Hamiltonian vector field associated to $\M$ and $A$ so that  $A=\M$.
Then,  by Lemma~\ref{lem:AM},  this vector field takes the form of 
as in the previous proposition with  $X=id$ and $Y=id$ and $\alpha=0$.
\item Consider a general Hamiltonian vector field from Theorem~\ref{thmA}
and define $X=P_p\, A'^{-1} \M'$ and $Y=P_q\, A^{-1} \M$. Then
$\M Y=X'\M$ follows from property 2 in Theorem~\ref{thmA}, and we obtain 
a vector field as in the previous proposition taking $\alpha=1$. 
\item  For $\alpha=0$ and $X=id$ and $Y=id$ we obtain the best response
dynamics described in Section~\ref{sec:games}. 
\end{enumerate}
\end{remark}

\begin{proof}
Let us consider a time interval $[t_1,t_2]$ during
which the orbit of  (\ref{eq:genBR}) lies on a straight line, 
i.e. so that $t\mapsto \BR_p(q(t))$
and $t\mapsto \BR_q(p(t))$ are both constant (and single-valued) for all $t\in (t_1,t_2)$. 
Let us write $p=p(t_1)$, $q=q(t_1)$ and denote the values of 
$\BR_p(q(t)),\BR_q(p(t))$ for $t\in (t_1,t_2)$ by   $\BR_p(q)$ and $\BR_q(p)$.
Then for $t\in [t_1,t_2]$ the orbit of (\ref{eq:genBR}) is 
equal to 
$$p(t)=e^{-t}p + (1-e^{-t})(X\BR_p(q)+\alpha \bp)\,,\,\, q(t)=e^{-t}q + (1-e^{-t})(Y\BR_q(p)+\alpha \bq).$$
Let us  project this orbit by the map $\pi$  onto $H^{-1}(H(p,q))$ and write $\tilde p(t)=\pi p(t)$
and $\tilde q(t)=\pi q(t)$.
Then
$$\tilde p(t)=(1-s(t))p(t) + s(t) \bp\, ,\,\,
\tilde p(t)=(1-s(t))q(t) + s(t) \bq$$
with $s(t)\in \R$ chosen so that the curve remains within a level set of the function $H$.
This means that
\begin{equation*}\begin{array}{ll}
&\BR_p(q)'\, \M \left[(1-s(t))\left( e^{-t}q + (1-e^{-t})(Y\BR_q(p)+\alpha \bq)\right) + s(t) \bq \right] - \\
& \\
&\quad  \left[ (1-s(t))\left( e^{-t}p + (1-e^{-t})(X\BR_p(q)+\alpha \bp)\right) + s(t) \bp \right]' \M \BR_q(p)  =\\
& \\
&\quad =  \BR_p(q)'\, \M q - p' \M \BR_q(p).\end{array}
\end{equation*}
Note that $\bp' \M \BR_q(p)=\bp' \M \bq=\BR_p(q)'\, \M \bq$ and 
$\M Y=X'\M$ 
and therefore the previous equation  is equivalent to 
$$(1-s(t)) e^{-t}  \left[ \BR_p(q)'\, \M q  - p' \M \BR_q(p) \right]
=  \BR_p(q)'\, \M q - p' \M \BR_q(p),$$
i.e., to $1-s(t)=e^t$ and $s(t)=1-e^t$. 
So the orbit becomes 
\begin{equation*}\begin{array}{ll}
\tilde p(t) \! &=e^t p(t) + (1-e^t)\bp=
e^t e^{-t}p + (e^t-1)(X\BR_p(q)+\alpha p) + (1-e^t)\bp\\
& \\
&=p + (e^t-1) (X\BR_p(q) +(\alpha-1) \bp)\quad \mbox{and } \\
& \\
\tilde q(t)& =q + (e^t-1) (Y\BR_q(p) +(\alpha-1) \bq).
\end{array}
\end{equation*}
Since this holds on all pieces of orbits, we get 
$$\dfrac{d\tilde p}{dt}=e^t(X\BR_p(\tilde q)\,+(\alpha-1)\bp) \, , \,\, \dfrac{d\tilde q}{dt}=e^t(Y\BR_q(\tilde p)+(\alpha-1)\bq).$$
If we use the time-parametrization $s=e^t$ this gives (\ref{eq:inducedflow})
which obviously is a translation flow, because the right hand side is piecewise
constant. 
\end{proof}

\section{The proof of the two Main Theorems}
Part 1 of the Main Theorem is proved in Proposition~\ref{prop:unique1}.
Parts 2,3 and 5 are proved in Theorem~\ref{thmA}.
Finally, part 4 follows from  Theorem~\ref{thm:trans}, Proposition~\ref{prop:projection}
and the remark below this proposition.
Theorem~\ref{thm:trans}, Proposition~\ref{prop:projection} also prove
the 2nd Main Theorem.

\section{A Hamilton vector field with random walk behaviour}\label{sec:randomwalk}
 The Hamiltonian vector fields which appear in this way can have rather
remarkable behaviour. We illustrate this by considering an example in the next theorem.
This example is part of a family of systems studied in \cite{SS09} in the context
of dynamical systems associated to game theory. 
Consider 
\begin{equation}
B= \left(\begin{array}{ccc} 1 & 0 & \beta \\ \beta & 1 & 0 \\ 0 & \beta & 1
\end{array} \right) 
\label{eq:example}
\end{equation}
where $\beta\ne -1,1$
and let $\M$ be some matrix close to $B$. Note that $B$ satisfies the transversality
assumption  (\ref{eq:trans2}) 
(here we use that $\det(B)=1+\beta^3$). Note that  $B$
has a unique completely mixed Nash equilibrium $(\bar p,\bar q)\in \Sigma\times \Sigma$
(the notion of completely mixed Nash equilibrium was defined in Definition (\ref{def:compmixed}); it corresponds to a vector with all components equal to $1/3$.

\subsection{A Hamilton vector field with random walk behaviour}
\label{sec:example}

As in Theorem~\ref{thmA}, let $X_H$ the Hamiltonian vector field corresponding
to a matrix $M$ close to the above matrix $B$ and the symplectic two-form $\omega=\sum_{ij}a_{ij}dp_i\wedge dq_j$ where $(a_{ij})$ are the coefficients of $A:=\M$. The flow of this  vector field acts as a random walk:

\begin{theorem}
\label{thmB}
Let $H$ be the Hamiltonian associated to a matrix $\M$ as in (\ref{eq:hamilton}), 
where $\M$ is sufficiently close to the matrix $B$ defined
in (\ref{eq:example}) where $\beta$ is taken to sufficiently close to the golden mean  
$\sigma:=(\sqrt{5}-1)/2\approx 0.618$.

Then the following holds. The set $H^{-1}(1)$ is homeomorphic to $S^3$ (it consists of pieces of hyperplanes). The flow $\phi_t$  of $X_H$ has a periodic orbit $\Gamma$ with the following properties.
$\Gamma$ is a hexagon in $H^{-1}(1)\subset \R^4$.
If one takes the first return map $F$ to a section $Z\subset H^{-1}(1)$ 
transversal to $\Gamma$
(i.e. a two-dimensional surface in $H^{-1}(1)$ containing some $x^0\in \Gamma$),   
then the following holds.
\begin{itemize}
\item For each $k\in \N$,  
there exists a sequence of periodic points $x_n\in Z$, $n=1,2,\dots$, of exactly period $k$
of the first return map to $Z$  converging to $x^0$ as $n\to \infty$.
\item The first return map $F$ to $Z$ has infinite topological entropy.
\item The dynamics acts as a random-walk. 
More precisely, there exist pairwise disjoint annuli $A_n$ in $Z$ (around $x^0$  so that 
$\cup_{n\ge 0} A_n\cup \{x^0\}$ is a neighbourhood of $x^0$ in $Z$) shrinking geometrically to $x_0$  so that  for each sequence $n(i)\ge 0$ with 
$|n(i+1)-n(i)|\le 1$ 
there exists a point $z\in Z$ so that $F^i(z)\in A_{n(i)}$ for all $i\ge 0$.
\end{itemize}
\end{theorem}

\bigskip
One obvious consequence of the random walking described in the theorem, is that the
stable and unstable manifold of the periodic orbit $\Gamma$ of the flow are rather unusual.
More precisely, take $\epsilon>0$ small,  let $\tau>0$, and define the
local stable set corresponding to rate $\tau$ as 
$$W^{s,\tau}_\epsilon(\Gamma):=\{x; \dist(\phi_t(x),\Gamma)\le \epsilon \mbox{ for all }t\ge 0\,\,\mbox{ and }\,\, 
\lim_{t\to \infty} \dfrac{1}{|t|} \log(\dist(\phi_t(x),\Gamma))\to \tau\}\quad $$
and the  local unstable set corresponding to rate $\tau$ as 
$$W^{u,\tau}_\epsilon(\Gamma):=\{x;  \dist(\phi_t(x),\Gamma)\le \epsilon \mbox{ for all }t\le 0 \,\,\mbox{ and } \,\,
\lim_{t\to -\infty} \dfrac{1}{|t|} \log(\dist(\phi_t(x),\Gamma))\to \tau\}.$$
Then the above system has for each $\epsilon>0$ and 
{\bf each} $\tau\ge 0$ close to zero,  that both $W^{s,\tau}_\epsilon(\Gamma)$ and 
$W^{s,\tau}_\epsilon(\Gamma)$ are non-empty in any neighbourhood of  $\Gamma$. 
This behaviour is completely different from that for  smooth differential equations in a three dimensional manifold with a hyperbolic periodic orbit: in that case, the stable 
(resp. unstable) sets are smooth manifolds and have the property that orbits converge to the periodic orbit in forward (resp. backward) time with a unique rate. In the above example, on the contrary,
for {\bf each} $\tau$  close to  zero, there are forward orbits which converge to $\Gamma$ with precisely rate $\tau$.

\medskip
For each $k\in \N$, the periodic orbits $x_n$ of period $n$ for the first return map 
mentioned in the first part of Theorem~\ref{thmB}, correspond to closed curves
$\gamma_n$ with period $T_n$ for which 
$$\gamma_n \to \Gamma, \,\, T_n\to kT\mbox{ as }n\to \infty.$$ In particular,  the set of periods of periodic orbits of the flow is certainly not discrete.

\subsection{Ouline of the proof of Theorem~\ref{thmB}}
Let us give an informal outline of the proof of Theorem~\ref{thmB} (full details can be found in 
\cite{SS09}). To do this, let us first  describe orbits of the flow $\phi_t$ near $\Gamma$.  Take local coordinates
in which $\Gamma\subset \{0\}\times \R\subset \R^2\times \R$ near some point $x_0$
of $\Gamma$. 
Let  $||z||=|z_1|+|z_2|$ be the  $l_1$ norm of $z=(z_1,z_2)\in \rz^2$,
\,\,  and let  $R_t$ be a rotation in $\R^2$ 
over angle $t$ leaving the 'circles'  in the $l_1$ norm
invariant (i.e. $||R_t(z)||=||z||$ for all $t$).
Then for $(z,0)\in \R^2\times \R$ and $t$ close to zero, one approximately has
$$\phi_t(z,0)=(R_{t/||z||},t)$$
and so during time $t$, the orbit spirals around $\Gamma$ approximately $t/||z||$ times.
In other words, the closer the orbit is to $\Gamma$, 
the tighter and faster the orbits spiral around $\Gamma$, see the left  two diagrams in  Fig~\ref{fig:coiling}. More precisely, the number of times the orbit
winds around $\Gamma$ during a time interval $t\in [0,1]$ is of the order
$1/||z||$ where $||z||$ is the distance of $z$ at time $t=0$. (Clearly, this
is only possible because the flow is not smooth.)

\begin{figure}[htp]
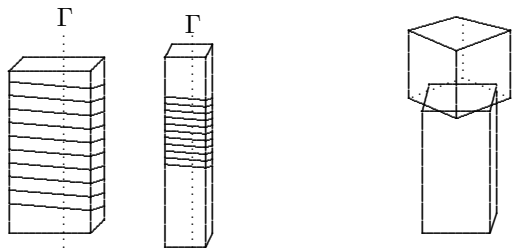
 \hfil
\beginpicture
\dimen0=0.18cm
\dimen1=0.09cm
\setcoordinatesystem units <\dimen0,\dimen0> point at -30 0
\setplotarea x from -4 to 4, y from -8 to 8
\setdots <1mm>
\plot 1 -7  1 9 /
\put {{$\Gamma$}} at 1.1  10
\setsolid
\plot -3 -6 3 -6 3 6 -3 6 -3 -6  /
\plot 3 -6 4 -5 4 7 3 6 /
\plot -3 6 -2 7 4 7 /
\setsolid
\plot -3 5.2 3 4.7 4 5 /
\plot -3 4.2 3 3.7 4 4 /
\plot -3 3.2 3 2.7 4 3 /
\plot -3 2.2 3 1.7 4 2 /
\plot -3 1.2 3 .7 4 1 /
\plot -3 .2 3  -0.3 4 0 /
\plot -3 -.8 3  -1.3 4 -1 /
\plot -3 -1.8 3  -2.3 4 -2 /
\plot -3 -2.8 3  -3.3 4 -3 /
\plot -3 -3.8 3  -4.3 4 -4 /
\setcoordinatesystem units <\dimen1,\dimen1> point at -80 0
\setplotarea x from 0 to 10, y from -8 to 8
\setdots <1mm>
\plot 1 -14  1 18 /
\put {{$\Gamma$}} at 1.1  19
\setsolid
\plot -3 -14 3 -14 3 14 -3 14 -3 -14  /
\plot 3 -14 4 -12 4 16 3 14 /
\plot -3 14 -2 16 4 16 /
\setsolid
\plot -3 8.2 3 7.7 4 8 /
\plot -3 7.2 3 6.7 4 7 /
\plot -3 6.2 3 5.7 4 6 /
\plot -3 5.2 3 4.7 4 5 /
\plot -3 4.2 3 3.7 4 4 /
\plot -3 3.2 3 2.7 4 3 /
\plot -3 2.2 3 1.7 4 2 /
\plot -3 1.2 3 .7 4 1 /
\plot -3 .2 3  -0.3 4 0 /
\plot -3 -.8 3  -1.3 4 -1 /
\plot -3 -1.8 3  -2.3 4 -2 /
\setcoordinatesystem units <\dimen1,\dimen1> point at -120 0
\setplotarea x from 0 to 10, y from -8 to 8
\plot -5 6 -5 -12  5 -12  5 6 -5 6 /
\plot -5 6 -4 10  6 10 5 6 /
\plot 6 9 6 -9 5 -12 /
\plot -7 8  0 5 8 8 /
\setdots <1mm>
\plot 8 8 1 11 -7 8 /
\plot 1 11 1 21 /
\setsolid
\plot -7 8 -7 18 0 15 0 5  /
\plot 0 15 8 18 8 8 /
\plot 8 18 0 20 /
\plot -7 18 0 20  /
\endpicture
\caption{\label{fig:coiling}
{\small  
In the left two figures,  orbits are drawn near a piece of $\Gamma$. Orbits spiral along
rectangular tubes around $\Gamma$. Orbits starting nearer 
to $\Gamma$ spiral with a finer pitch around $\Gamma$, see the 2nd figure on the middle.
Locally orbits spiral on these rectangular tubes, but it is certainly not true that 
orbits remain on topological tori around $\Gamma$, because the tube is skewed when it has fully
followed the entire orbit $\Gamma$, see the figure on the right.}}
\end{figure}

Moreover, the first return map   $P\colon Z\to Z$ near to $\Gamma$ has a very special form. 
If we identify $Z$ with $\rz^2$ and $\Gamma\cap Z$ with $0\in \rz^2$, then $P$
is  approximately a composition of maps of the form
\begin{equation}
F(z)= M\circ R_{1/||z||}(z) 
\label{eq:modelF}
\end{equation}
 where $R$ and $||z||$ are as above, and where $M$ is a matrix
 of the form $M=\left( \begin{array}{cc} 2 &  0 \\ 0 & 0.5 \end{array}\right)$.
 Note that $F$ is a homeomorphism which is non-smooth at $0$. 
 The image of a ray through $0$ under the map $F$ is a spiral, see the middle diagram in Figure~\ref{fig:fixedpoints}.
There exist $r\in (0,1)$ so that $F$ can send a point in the annulus $A_n:=\{z\in \rz^2; r^{n+1}\le ||z|| \le r^n\}$
to annuli $A_{n-1},A_n,A_{n+1}$ in a way which almost completely \lq forgets{\rq} 
where the point started from. So the dynamics can be essentially modelled by that of a random walk.
Indeed, it was proved in \cite{SS09} that there are orbits which can go to $0$
at a prescribed speed. The image of a set $\{z\in \rz^2; ||z||=1\}$ under the first six iterates of $F$
is drawn in Figure~\ref{fig:spiraling}.

\begin{figure}[htp] \hfil
\beginpicture
\dimen0=0.15cm
\setcoordinatesystem units <\dimen0,\dimen0> point at 20  0 
\setplotarea x from 0 to 20, y from -8 to 8
\plot -6 -3   0 -6   6 -3   9   3   3 6  -3 3 -6 -3 /  
\setdashes
\plot 4 -2 4 2 9 2 9 -2 4 -2 / 
\setsolid 
\plot 4 0 4 2 9 2 9 -2 7 -2 / 
\put {{$Z$}} at 11 2
\put {{$D$}} at 0  1
\put {{$\Gamma$}} at 1.1  6.2
\setcoordinatesystem units <\dimen0,\dimen0> point at -10 0 
\setplotarea x from -20 to 20, y from -10 to 10
\setdashes
\plot -11 0 11 0 /
\plot 0 -10 0 10 /
\setsolid
\plot 10 0 0 9 -8 0 0 -7 6 0 0 5 -4 0 0 -3 2 0 /
\plot 0 0  8 8 /
\put {$l$} at 9 9 
\setcoordinatesystem units <\dimen0,\dimen0> point at -40 0 
\setplotarea x from -20 to 20, y from -10 to 10
\setdashes
\plot -13 0 13 0 /
\plot 0 -12 0 13 /
\setsolid
\plot 10 0 0 10 -10 0 0 -10 10 0 /
\plot 9.1 0 0 9.1 -9.1 0 0 -9.1 9.1 0 /
\plot 8.3 0 0 8.3 -8.3 0 0 -8.3 8.3 0 /
\plot 7.6 0 0 7.6 -7.6 0 0 -7.6 7.6 0 /
\plot 7 0 0 7 -7 0 0 -7 7 0 /
\plot 6.5 0 0 6.5 -6.5 0 0 -6.5 6.5 0 /
\endpicture
\caption{\label{fig:fixedpoints}{\small On the left, the orbit $\Gamma$
with a transversal section $Z$. $\Gamma$ is a hexagon in $\R^4$, 
and is here presented as a genuine hexagon. 
In the middle, the image of a half-ray $l$ through $0$ under the model map $F$ defined in (\ref{eq:modelF}) is drawn. 
A conveniently chosen curve $l$ contains a sequence of fixed points, converging to $0$.
On the right, a nested sequence of annuli $A_n$ in $Z$ is drawn. There exist orbits which visit annuli $A_{n(i)}$
where $n(i)$ can be chosen arbitrarily so that $|n(i+1)-n(i)|\le 1$ and $n(i)\ge 0$.}}
\end{figure}

\begin{figure}[htp] \hfil
\includegraphics[width=6in]{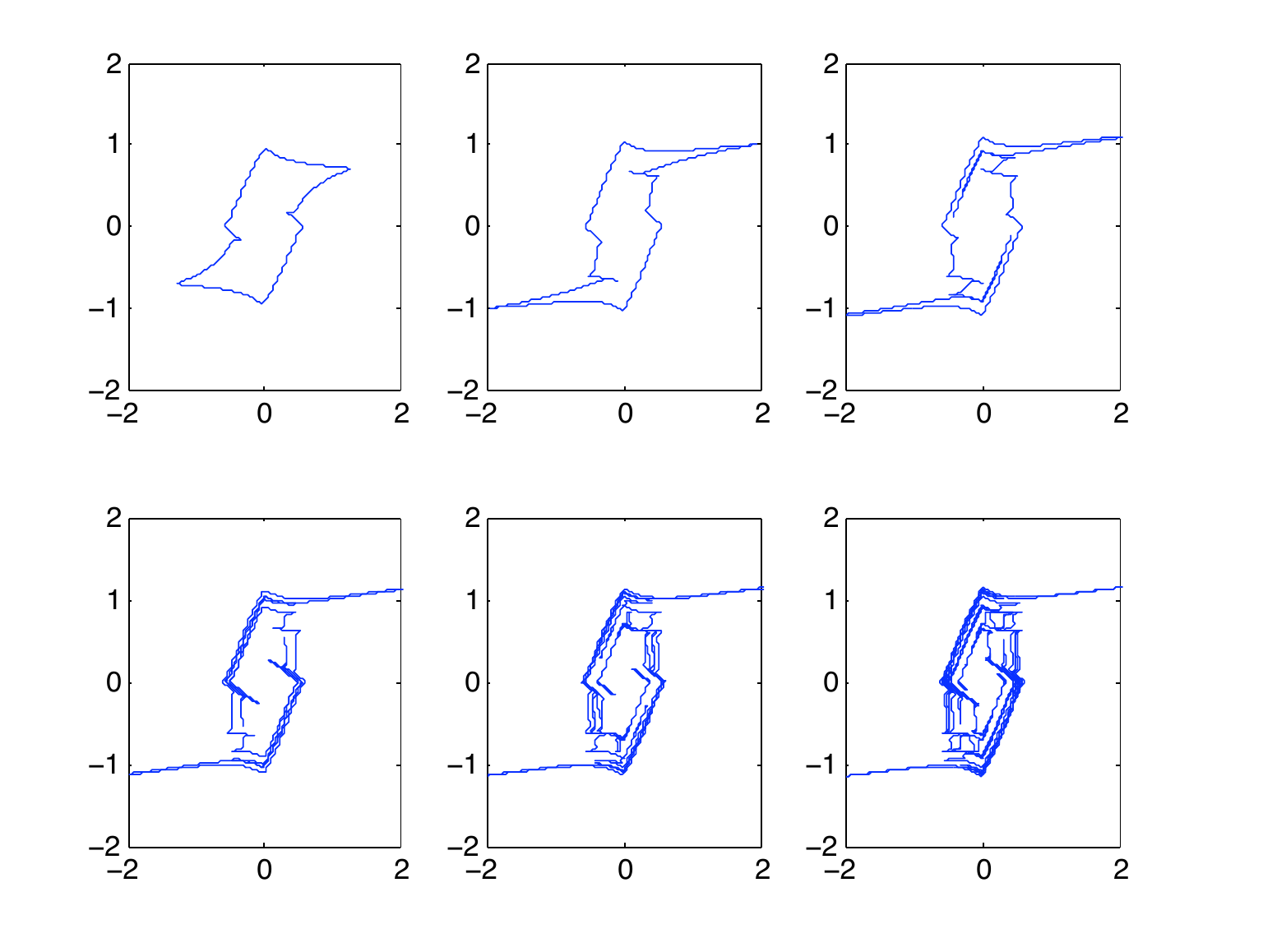}
\caption{\label{fig:spiraling} The image of under the first six iterates under  $x\mapsto M\circ R_{1/||x||}(x)$
of the set $||x||=1$.}
\end{figure}

\subsection{The existence of a  global section}
In the example considered in the previous theorem, $H^{-1}$
is homeomorphic to $S^3$, i.e. to $\R^3\cup \{\infty\}$. It turns out that the 
flow has a global first return section $S$, namely a topological disc spanned by 
the periodic orbit $\Gamma$.

\begin{theorem}[Existence of a global first return section through $\Gamma$]
\label{thmB2}
For the Hamiltonian differential inclusion defined 
in Theorem~\ref{thmB}, there exists a topological disc $S\subset H^{-1}(1)\approx S^3$ whose boundary  is equal to the periodic orbit $\Gamma$
so that the following properties hold: 
\begin{enumerate}
\item The orbit through each point $x\in H^{-1}(1)\setminus \Gamma$ 
intersects the topological disc $S$ infinitely many times (and each of these intersections
is transversal). So $S$ is a global section with a well-defined first return map $R_S$.
\item The first return {\em time} of $z$ to $S$ tends to zero when $z\in S\setminus \Gamma$ tends to $z'\in \Gamma$. 
\item  The first  return map $R_S$ to $S$ has a continuous extension to $\partial S$
and $R_S|\partial S=id$.
\item One can choose $S$ to be the union of four (two-dimensional) 
triangles in $\R^4$ and so that 
$R_S$ is piecewise affine and area preserving.
\item $R_S$ has homoclinic intersections of periodic orbits, and hence horseshoes. 
\end{enumerate}
\end{theorem}

Of course, a smooth vector field could never have such a global section.

\subsection{Visualising the dynamics in $S^3$.}
The existence of the global section $S$ makes it rather easy to visualize
the flow. The set $S$ is a union of four triangles in $\R^4$, and 
a projection of this set is drawn in  Figure~\ref{fig:S-triangles}.
As is shown in \cite{SS09} the following holds.
\begin{enumerate}
\item The flow on $H^{-1}(1)\approx S^3=\R^3\cup \{\infty\}$ has
exactly one periodic orbit which meets the section $S$ once.
This periodic orbit is of saddle-type. 
\item  The flow has exactly one periodic orbit  $c$ 
which meets the section $S$ twice.
This periodic orbit $c$ is elliptic and intersects $S$ in the centers of the two {\lq}egg-like{\rq}  regions shown 
in $S$. The {\lq}egg-like{\rq}  regions are filled by 
invariant circles (corresponding to invariant tori for the flow). This can be seen 
as follows: $c \cap D$ consists of two points $p_1,p_2$. 
As mentioned, the first return map $R_S$ to $S$ is a piecewise affine map.
Moreover,  the linearisation of $R_S^2$ at $p_1$
has two complex conjugate eigenvalues on the unit circle (not equal to one),
see \cite{SS09}. It follows that $R_S^2$ is linearly conjugate
to a rotation, and so orbits near $x_0$ lie on ellipses.
\end{enumerate}

As mentioned, in the previous the previous theorem,
$\Gamma$ is a periodic orbit of the flow on the boundary of $S$, and each other orbit  
of the flow in $H^{-1}(1)\approx S^3=\R^3\cup \{\infty\}$ 
transversally intersects the section $S\setminus \Gamma$ infinitely many times.

\subsection{Open problems}
Many questions are still open about this specific system. For example, 
from the above description it follows that
$H^{-1}(1)$ is the union of two fully invariant sets $\mathcal U$ and $\mathcal W$
where $\mathcal U$ consists of a closed solid torus which intersects $S$ in the two {\lq}egg-shaped{\rq}
regions and where $\mathcal W$ is the complement (so $\mathcal W$ is the interior of a solid torus).

\begin{question}
Does there exist an orbit in $\mathcal W$ which is dense in $\mathcal W$? 
\end{question}

If one considers the first return map to $S$ one can ask:

\begin{question}
Consider the first return map $R_s\colon S\to S$.
\begin{enumerate}
\item Are there any other elliptic orbits (of higher period)?
\item Take a (Lebesgue typical) point $x\in \mathcal W$. What can be said
about the limits of 
$\dfrac{1}{n} \sum_{i=0}^{n-1} \delta_{R^i(x)}$? Does it exist? Is it unique? 
Is it absolutely continuous? 
\item Is it possible to find an adequate random-walk description 
of the first return map to a section $\Sigma$? 
\end{enumerate}
\end{question}

It would also be interesting to 
understand the dynamics of the map 
 $F(x)=M\circ R_{1/||x||}(x)$ as in (\ref{eq:modelF}) more thoroughly, where $R_\theta(x)$
 is a rotation and $M$ is a linear map of saddle-type. For example, 
 are typical orbits dense in $\R^2$ for generic matrices $M$?
 For some results on this, see \cite{SS09}. 
\medskip

More general questions are posed in the conclusion,
see Section~\ref{sec:conclusion}.

\begin{figure}[htp] \hfil
\includegraphics[width=3in]{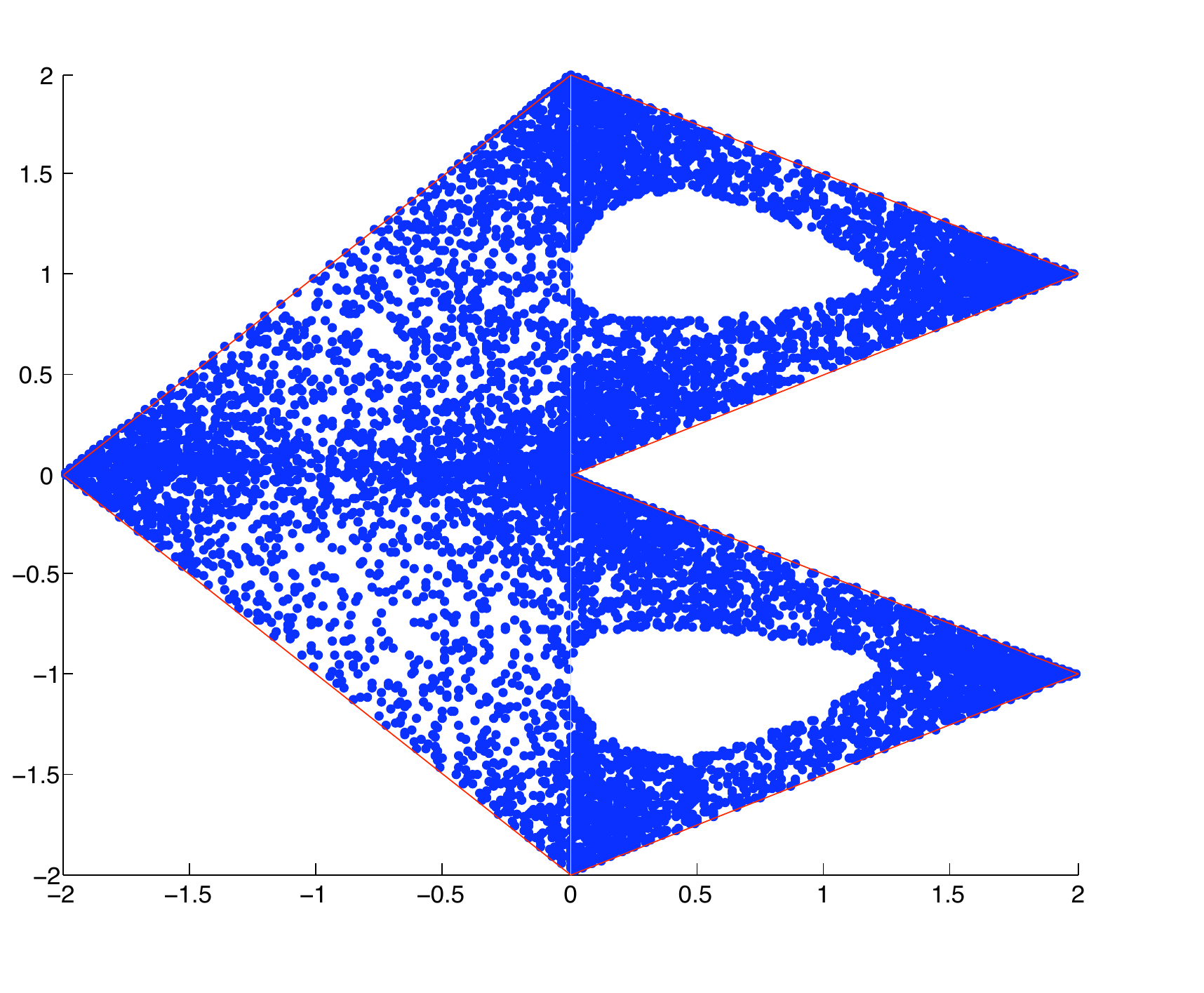}
\caption{\label{fig:S-triangles}A projection of the global section $S$.
Note that $S$ is topological disc whose boundary is a hexagon $\Gamma$
where $\Gamma$ is a periodic orbit of the flow. Orbits spiral around $\Gamma$ 
as in Figure~\ref{fig:coiling}.  Note that we have drawn the hexagon differently
than in Figure~\ref{fig:fixedpoints} (the reason for drawing $S$ in this way is that
it represents better how $S$ is naturally embedded in $H^{-1}(1)=S^3$
as is explained in Section 3 of \cite{SS09}).  We also plot
the orbit of a typical starting point under the first return map to $S$.  There are
two 'egg-shaped' regions which are permuted by the first return map. Orbits within this
region form invariant circles (these 'egg-shaped' regions have period two). At the center
of these regions there is a periodic orbit of period two. Locally, the map
is a rigid rotation over an angle which is determined by the parameter $\beta$.
For $\beta=\sigma$ the rotation angle is an irrational number, so orbits under the first return map
lie on  circles.}
\end{figure}

\bigskip

\section{Application to zero sum games}
\label{sec:games}
In this section we will apply the previous results to 
to fictitious play and best response dynamics defined in game theory
and state some applications of our results which are relevant to game theory.

Let us first give a short introduction into the relevant notions from game theory.
Consider a two-player game where player $P$ has a choice of $m$ actions
and player $Q$ has a choice of $n$ actions for each time $t\in [1,\infty)$.
So let $\Sigma_p\subset \rz^m$ and $\Sigma_q \subset \rz^n$ be the 
space of  probability vectors in $\rz^m$ resp. 
$\rz^n$. 
For each $t\in [0,\infty)$,  each of the players continuously chooses an action 
$$m^P(t)\in \{e_1^p,\dots,e_m^p\}\subset \Sigma_p \mbox{ and } m^Q(t)\in \{e_1^q,\dots,e_n^q\}\subset \Sigma_q.$$
Let 
\begin{equation}
p(t)=\dfrac{1}{t} \int_{s=1}^t m^P(s)\, ds\mbox{ and }
q(t)=\dfrac{1}{t} \int_{s=1}^t m^Q(s)\, ds.\label{eq:meanp}
\end{equation}
Hence $p(t)\in \Sigma_p$ describes the average of the strategies  player $P$
has chosen during the time interval $[1,t)$. Usually $p(t)$ and  $q(t)$ are called the
strategy profiles of players $P$ and $Q$.

In this model, it assumed that player $P$ only observes (or responds to)  $q(t)$ and tries to choose an action which maximizes his payoff. In other words, it is assumed that there are best-response maps
(possibly multivalued)
$\Sigma_q\ni q\mapsto \BR_p(q)\in \Sigma_p$
and $\Sigma_p\ni p\mapsto \BR_p(p)\in \Sigma_q$.
One often makes the assumption that there are $n\times m$ matrices
$P,Q$ so that 
$$\BR_p(q):=\argmax_{p\in \Sigma_p} p'\, P q\mbox{ and }
\BR_q(p):=\argmax_{q\in \Sigma_q} p'\,  Q q.$$
It is then assumed that player $P$ chooses at time $t$ an action $m^P(t)$ for which
\begin{equation}
m^P(t)\in \BR_p(q(t))
\label{bra}
\end{equation}
while $Q$ chooses an action 
\begin{equation}
m^Q(t)\in \BR_q(p(t)).
\label{brb}\end{equation}
Differentiating (\ref{eq:meanp})  immediately leads to 
\begin{equation}
\dfrac{dp}{dt}= \dfrac{1}{t}(m^P(t)\,-p(t)), \quad \dfrac{dq}{dt}\in \dfrac{1}{t}(m^Q(t)-q(t))
\label{eq0}
\end{equation}
and combining this with (\ref{bra}) and (\ref{brb}) gives that
$t\mapsto (p(t),q(t))$ is a solution to the following differential inclusion
\begin{equation}
\dfrac{dp}{dt}\in \dfrac{1}{t}(\BR_p(q)\,-p), \quad \dfrac{dq}{dt}\in \dfrac{1}{t}(\BR_q(p)-q).
\label{eqfp}
\end{equation}
This is called {\em fictitious play dynamics}.
Some authors prefer a different time parametrization (taking $s=e^t$), which
gives
\begin{equation}
\dfrac{dp}{ds}=\BR_A(q)\,-p, \dfrac{dq}{ds}=\BR_B(p)-q.
\label{eqbr}
\end{equation}
and this is called {\em best response dynamics}.
As noted before,   it is not hard to show that 
these differential inclusions have solutions, see \cite{AubinCellina84}).
Best response dynamics is commonly associated with two infinite populations,
so that within each of these two populations the fraction of players choosing
a certain strategy continuously evolves towards best response, see \cite{MR1180001}.
A common interpretation of fictitious play is as a model for rational learning,
see for example Fudenberg and Levine \cite{Fudenberg-Levine98}.

In the present paper we consider the zero-sum case.
In this situation, the players have opposite interests and
so we have  $P=\M$ and $Q=-\M$.

A more detailed explanation of the rationale of this model can be found
in for example  the monograph of Fudenberg and Levine  \cite{Fudenberg-Levine98}. One reason this model is used widely is because it 
is frequently used as a learning model in economic theory.

Restating our previous results gives:

\begin{theorem}\label{thm:zerosumcont}
Assume that we have a $m\times n$ two-player zero-sum game
with $m$ and $n$ not necessarily equal.
Assume that the transversality condition (\ref{eq:trans2}) holds. Then 
\begin{enumerate}
\item the Nash equilibrium $(\bp,\bq)$ is unique;
\item (\ref{eqbr}) and (\ref{eqfp}) define continuous flows
converging to the Nash equilibrium $(\bp,\bq)$;
\item there exist $r\le \min(m,n)$ and $i_1,\dots,i_r\in \{1,\dots,m\}$
and $j_1,\dots,j_r\in \{1,\dots,n\}$ so that
$\bp\in  [ e_{i_1},\dots,e_{i_r}]\subset \Sigma_p$
and $\bq\in  [ e_{j_1},\dots,e_{j_r}] \subset \Sigma_q$
and so that moreover the $i_1,\dots,i_r$-th component  of $\bp$ and 
$j_1,\dots,j_r$-th component of $\bq$ 
are non-zero (in other words, $(\bp,\bq)$ is a completely mixed Nash equilibrium w.r.t.
to this subgame);
\item points starting in $[ e_{i_1},\dots,e_{i_r}]\times  [ e_{j_1},\dots,e_{j_r}]$
stay in this subspace under the flow;
\item the other components of $p,q$  decrease monotonically
with time for any solution of the differential equations (\ref{eqbr}) and (\ref{eqfp}).
\item each half-line through $(\bar p,\bar q)$ has a unique intersection
with $H^{-1}(1)$; if we consider the flow restricted to 
$[ e_{i_1},\dots,e_{i_r}]\times  [ e_{j_1},\dots,e_{j_r}]$
and project this onto $H^{-1}(1)$ we obtain a Hamiltonian flow.
\end{enumerate}
\end{theorem}

That solutions converge to the Nash equilibrium was already proved 
in the 1950's by Robinson~\cite{Robinson51}. According to the previous theorem, 
generically there is  precisely one Nash equilibrium and this is a completely
mixed Nash equilibrium for some $r\times r$ subgame (which is an invariant subset
of the dynamics). 

\bigskip

In the trivial case of a $2\times 2$ zero-sum 
game, orbits spiral towards the Nash equilibrium. If we take $H$ as
before, then we obtain an induced Hamiltonian flow
on $H^{-1}(c)$, $c>0$. It turns out that $H^{-1}$ is  the boundary of 
a quadrilateral and the Hamiltonian flow moves either clockwise or counter clockwise on this boundary,
see Figure~\ref{fig:zerosum}.

In the higher dimensional case of a $k\times k$ game
with $k\ge 3$, typical orbits will be more complicated. 
Indeed, since a Hamiltonian  flow preserves volume, it follows 
from the Poincar\'e recurrence theorem that Lebesgue almost every point returns to 
a neighbourhood of the initial point, see \cite{MR648108}.
In the example we presented, for typical starting points, players switch strategies  erratically.

The examples considered in Theorem~\ref{thmB}  are relevant to a result
of  Krishna and Sj\"ostr\"om (see \cite{Krishna-Sjostrom98}). Their result deals with orbits with cyclic play. 
More precisely,  one has cyclic play  along the orbit $t\mapsto (p(t),q(t))$ if the values of 
$t\mapsto (\BR_p(q(t)),\BR_q(p(t))$ are changing periodically with some period $s\in \{1,2,\dots,\}$, i.e.,   if
there exists times $t_0<t_1<t_2<t_3<t_4<\dots$ so that for all $i\ge 0$ and all $t\in (t_i,t_{i+1})$,  \, 
$(\BR_p(q(t)),\BR_q(p(t))$  is equal to some corner point in $V_i\in \Sigma_p\times \Sigma_q$
and so that  
$V_{i+s}=V_i$  for all $i\ge 0$. In other words, both players repeat strategies every $s$-th step.
The  theorem of Krishna and Sjostrom states that for generic games, it is impossible for an open set of initial conditions
to all converge to the Nash equilibrium with the same cyclic play. 
%
%
%
%
%
The next corollary shows that their theorem does NOT hold for non-generic games
(in particular not for zero sum games):

\begin{corollary}
There exists an open set of matrices, 
so that the corresponding zero-sum games have the following property:
for an open set of initial conditions one has
convergence  to the Nash equilibrium with cyclic play. 
\end{corollary}
\begin{proof}
Consider the zero-sum game associated to the matrix (\ref{eq:example})
and take an orbit corresponding to an initial condition
which lies in one of the two egg-shaped region  in Figure~\ref{fig:S-triangles}.
Here the region bounded by the six sides forms the global section $S$ from 
Theorem~\ref{thmB2} associated to the flow induced on $H^{-1}(1)$ and the {\lq}centres{\rq} 
of the two egg-shaped regions in Figure~\ref{fig:S-triangles}
correspond to a periodic two orbit $\gamma$ of the first return map to $S$.
Let $D\cap \gamma=\{x_0,x_1\}$ and let $F$ be the second iterate
of the first return map to $S$. Then $F(x_0)=x_0$ and it is shown in \cite{SSH08} and \cite{SS09} that the 
linearisation of $F$ at $x_0$  has eigenvalues $\lambda,\bar \lambda$ with $|\lambda|=1$.
This orbit corresponds to the Shapley orbit in \cite{SSH08} and \cite{SS09}
and as was shown there,  the periodic orbit $\gamma$ crosses 
the indifference sets exactly 6 times (each time transversally). 
Now consider an orbit $\tilde \gamma$
near $\gamma$.
%
%
%
%
%
%
%
Since the first return map $R_S$ to $S$ is a piecewise translation, 
near $x_0$ the second iterate $R_s^2$ is a  (smooth) rotation. Here we use that
$\gamma$ meets all indifference sets (sets where $\BR_p$ or $\BR_q$ are multivalued) 
transversally. It follows that $\tilde \gamma\cap D$ lies on two circles around $x_0$ resp. $x_1$
(provided $\tilde \gamma$ is sufficiently close to $\gamma$). 
It also follows that $\tilde \gamma$ hits all the indifference sets transversally and in the same order as
$\gamma$. (But of course the times at which $\tilde \gamma$ hits the indifference planes need not necessarily be periodic.) In any case, all orbits $\tilde \gamma$ 
sufficiently near $\gamma$ have periodic play. 

All this also holds for all zero-sum games near the one considered
in Theorem~\ref{thmB}. Indeed, the eigenvalues of the linearization of $R_S^2$ at $x_0$
both lie on the unit circle.  Since these eigenvalues are not equal to one, they are complex conjugate.
Hence for a nearby game the corresponding map $\tilde R_S^2$ also has a fixed point at a point $\tilde x_0$ 
close to $x_0$, and the linearization of $\tilde R_S^2$ at $\tilde x_0$ also has two eigenvalues 
which are complex conjugate. Since $\tilde R_S^2$ is area preserving, the two eigenvalues again
lie on the unit circle. 
\end{proof}

\begin{figure}[htp]
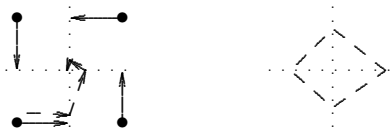
 \hfil
\beginpicture
\dimen0=0.2cm \dimen1=0.07cm
\setcoordinatesystem units <\dimen1,\dimen1> point at 0 0
\setplotarea x from -10 to 10, y from -10 to 10
\setlinear \setdots
\plot -12 0 12 0 /
\plot 0 12 0 -12 / \setsolid
\put {$\bullet$} at 10 10
\put {$\bullet$} at -10 10
\put {$\bullet$} at -10 -10
\put {$\bullet$} at 10 -10
\arrow <5pt> [0.2,0.4] from 10 10 to 0.5 10
\arrow <5pt> [0.2,0.4] from -10 10 to -10 0.5
\arrow <5pt> [0.2,0.4] from -10 -10 to -0.5 -10
\arrow <5pt> [0.2,0.4] from 10 -10 to 10 -0.5
\setdashes
\arrow <5pt> [0.2,0.4] from -8 -8 to 0 -8.5
\arrow <5pt> [0.2,0.4] from 0 -8.5 to 3 0
\arrow <5pt> [0.2,0.4] from 3 0 to 0 2
\arrow <5pt> [0.2,0.4] from 0 2 to -0.5 0
\setcoordinatesystem units <\dimen1,\dimen1> point at -50 0
\setplotarea x from -10 to 10, y from -10 to 10
\setlinear \setdots
\plot -12 0 12 0 /
\plot 0 12 0 -12 / \setsolid
\setdashes
\plot 10 0 0 8 -8 0 0 -7 10 0 /
\endpicture
\caption{\label{fig:zerosum}
{\small A $2\times 2$ zero-sum game converges to the Nash equilibrium. The level set $H^{-1}(c)$
is a quadrilateral and the Hamiltonian flow moves periodically on this.}}
\end{figure}

\subsection{Non-zero sum games}
Even in the case when the players are involved in a general sum game, the above
differential inclusion (\ref{eqbr}) makes sense. There are many papers which show that one has 
convergence to the equilibrium in particular situations:  
for  games where one or both of the players have only 2 strategies to choose from,
see \cite{Miyasawa61} and  \cite{Metrick-Polak94}
for the $2\times 2$ case;  \cite{Sela00} for the $2\times 3$ case; and
\cite{Berger05} for the general $2\times n$  case.
A $2\times 2 \times 2$ fictitious game with a stable limit cycle was constructed in \cite{Jordan93}.  
The $3\times 3$ example studied in this paper,  shows that the situation is far more complicated in 
general, see \cite{SSH08} and \cite{SS09}. 
Theorem~\ref{thmB} summarises some of the results from these last two papers.

\section{Conclusion}
\label{sec:conclusion}

In this paper we saw that the Hamiltonian function $H$ defined
as in  (\ref{eq:hamilton}) generates a Hamiltonian vector field
(which is piecewise constant) with a continuous translation flow. 
We also described an example, showing that the dynamics
of such systems can be surprisingly rich.
Even though
many questions remain about the case considered, it
may be a good idea to consider more elementary
questions in general. For example, the following question is very natural,
and seems wide open.

\begin{question}
Assume that $H$ is as in (\ref{eq:hamilton}) where $\M$ satisfies
all the transversality conditions. 
\begin{enumerate}
\item Does the flow necessarily have periodic orbits?
\item Are there always infinitely many periodic orbits?
\item Can the set of periods of periodic orbits be discrete?
\item Are orbits dense in certain open sets?
\end{enumerate}
\end{question}

\bigskip
In the vain of the Palis conjecture, see \cite{MR2399817}, one can also ask questions about the
statistical behaviour of orbits:

\begin{question}[A version of the Palis conjecture]
Assume that $H$ is as in (\ref{eq:hamilton}) where $\M$ satisfies
all the transversality conditions. 
\begin{enumerate}
\item Can these systems be ever structurally stable (within the
class of systems under consideration)?
\item What are the possible physical measures for these systems? 
\end{enumerate}
\end{question}

\bibliographystyle{plain}
\bibliography{../games_bib}

\end{document}

%% file: prepicte.tex
\catcode`@=11 \catcode`!=11

\expandafter\ifx\csname fiverm\endcsname\relax
  \let\fiverm\fivrm
\fi
  
\let\!latexendpicture=\endpicture 
\let\!latexframe=\frame
\let\!latexlinethickness=\linethickness
\let\!latexmultiput=\multiput
\let\!latexput=\put
 
\def\@picture(#1,#2)(#3,#4){%
  \@picht #2\unitlength
  \setbox\@picbox\hbox to #1\unitlength\bgroup 
  \let\endpicture=\!latexendpicture
  \let\frame=\!latexframe
  \let\linethickness=\!latexlinethickness
  \let\multiput=\!latexmultiput
  \let\put=\!latexput
  \hskip -#3\unitlength \lower #4\unitlength \hbox\bgroup}

\catcode`@=12 \catcode`!=12

%% file: postpict.tex
\catcode`@=11 \catcode`!=11
  
\let\!pictexendpicture=\endpicture 
\let\!pictexframe=\frame
\let\!pictexlinethickness=\linethickness
\let\!pictexmultiput=\multiput
\let\!pictexput=\put

\def\beginpicture{%
  \setbox\!picbox=\hbox\bgroup%
  \let\endpicture=\!pictexendpicture
  \let\frame=\!pictexframe
  \let\linethickness=\!pictexlinethickness
  \let\multiput=\!pictexmultiput
  \let\put=\!pictexput
  \let\input=\@@input   
  \!xleft=\maxdimen  
  \!xright=-\maxdimen
  \!ybot=\maxdimen
  \!ytop=-\maxdimen}

\let\frame=\!latexframe

\let\pictexframe=\!pictexframe

\let\linethickness=\!latexlinethickness
\let\pictexlinethickness=\!pictexlinethickness

\let\\=\@normalcr
\catcode`@=12 \catcode`!=12